\documentclass[a4paper,twoside,10pt]{article}

\usepackage[a4paper,left=3cm,right=3cm, top=3cm, bottom=3cm]{geometry}
\usepackage[latin1]{inputenc}
\usepackage{mathrsfs}
\usepackage{graphicx}
\usepackage{epstopdf}
\usepackage{amsmath}
\usepackage{amsthm}
\usepackage{algorithm,algpseudocode}
\usepackage{amssymb}
\usepackage{hyperref}
\usepackage{stmaryrd}
\usepackage{subfigure}
\usepackage{float}
\usepackage{bigints}
\usepackage{cite}
\usepackage{color}
\usepackage[abs]{overpic}
\usepackage[font=footnotesize,labelfont=bf]{caption}
\usepackage{cases}
\usepackage[author={Lorenzo}]{pdfcomment}
\usepackage{soul,xcolor}

\restylefloat{table}
\theoremstyle{plain}
\newtheorem{thm}{Theorem}[section]

\newtheorem{lem}[thm]{Lemma}

\theoremstyle{definition}

\theoremstyle{remark}
\newtheorem{remark}{Remark}

\newcommand{\vertiii}[1]{{\left\vert\kern-0.25ex\left\vert\kern-0.25ex\left\vert #1  \right\vert\kern-0.25ex\right\vert\kern-0.25ex\right\vert}}

\newcommand{\diam}{\text{diam}} 
\newcommand{\E}{E} 
\newcommand{\e}{e} 
\newcommand{\n}{\mathbf n} 
\newcommand{\tauh}{\mathcal T} 
\newcommand{\Nu}{\mathcal V} 
\newcommand{\dof}{\text{dof}} 
\newcommand{\aux}{\text{aux}} 
\newcommand{\card}{\text{card}} 
\newcommand{\p}{p} 
\renewcommand{\u}{u} 
\renewcommand{\v}{v} 
\newcommand{\up}{u_\p} 
\newcommand{\vp}{v_\p} 
\newcommand{\zp}{z_\p} 
\newcommand{\zpz}{\zp^{(0)}} 
\newcommand{\vpmu}{v_{\p-1}} 
\renewcommand{\wp}{w_\p} 
\newcommand{\wpmu}{w_{\p-1}} 
\newcommand{\V}{V} 
\newcommand{\VE}{\V(\E)} 
\newcommand{\Vp}{\widetilde \V_\p} 
\newcommand{\Vptilde}{\V_\p} 
\newcommand{\Vpmu}{\widetilde\V_{\p-1}} 
\newcommand{\VpE}{\widetilde \V_\p(\E)} 
\newcommand{\VptildeE}{\V_\p(\E)} 
\renewcommand{\a}{a} 
\newcommand{\aE}{\a^\E} 
\newcommand{\ap}{\a_\p} 
\newcommand{\apmu}{\a_{\p-1}} 
\newcommand{\apE}{\a_\p^\E} 
\newcommand{\apmuE}{\a_{\p-1}^\E} 
\newcommand{\f}{f} 
\newcommand{\fp}{\f_\p} 
\newcommand{\gp}{g_\p} 
\newcommand{\Pinabla}{\Pi ^{\nabla}} 
\newcommand{\Pinablap}{\Pi ^{\nabla}_\p} 
\renewcommand{\S}{S} 
\newcommand{\SE}{\S^\E} 
\newcommand{\SEp}{\S^\E_\p} 
\newcommand{\StildeEp}{\S^\E_{\p,aux}} 
\newcommand{\Pizpmu}{\Pi^0_{\p-1}} 
\newcommand{\Pizpmd}{\Pi^0_{\p-2}} 
\newcommand{\Pizpmt}{\Pi^0_{\p-3}} 
\renewcommand{\c}{c} 
\newcommand{\chat}{\widehat \c} 
\newcommand{\cstab}{\c_{\text{\tiny{STAB}}}} 
\newcommand{\Ihat}{\widehat I} 
\newcommand{\q}{q} 
\newcommand{\Ap}{A_\p} 
\newcommand{\Lp}{L_\p} 
\newcommand{\Lpbold}{\mathbf {\Lp}} 
\newcommand{\Apbold}{\mathbf{A}_\p} 
\newcommand{\Apmu}{A_{\p-1}} 
\newcommand{\boldalpha}{\boldsymbol{\alpha}} 
\newcommand{\boldbeta}{\boldsymbol{\beta}} 
\newcommand{\upvect}{\mathbf u_\p} 
\newcommand{\uptildevect}{\widetilde{\mathbf u}_\p} 
\newcommand{\Ibold}{\mathbf I} 
\newcommand{\fpvect}{\mathbf f_\p} 
\newcommand{\qpmu}{{q_{\p-1}}} 
\newcommand{\epmu}{{e_{\p-1}}} 
\newcommand{\epmubar}{{\overline e_{\p-1}}} 
\newcommand{\rpmu}{ r_{\p-1}} 
\newcommand{\Ipmup}{I_{\p-1}^{\p}} 
\newcommand{\Ippmu}{I_{\p}^{\p-1}} 
\newcommand{\Pppmu}{P_\p^{\p-1}} 
\newcommand{\bbeth}{\zeta} 
\newcommand{\Id}{\text{Id}} 
\newcommand{\md}{m_2} 
\newcommand{\cdlvl}{\c_{\text{2lvl}}}
\newcommand{\Bp}{B_\p} 
\newcommand{\Bbold}{\mathbf B} 
\newcommand{\Gp}{G_\p} 
\newcommand{\Gpmd}{G_\p^{\md}} 
\newcommand{\Sigmapmd}{\Sigma_{\p,\md}} 
\newcommand{\Sigmapmumd}{\Sigma_{\p-1,\md}} 
\newcommand{\m}{m} 
\newcommand{\Kpbold}{\mathbf A _\p} 
\newcommand{\uI}{\u_I} 
\newcommand{\upi}{\u_{\pi}} 
\newcommand{\zbold}{\mathbf{z}} 
\newcommand{\gbold}{\mathbf{g}} 
\newcommand{\rbold}{\mathbf{r}} 
\newcommand{\ebold}{\mathbf{e}} 
\newcommand{\ebarbold}{\mathbf{\overline e}} 

 \newcommand{\Kbold}[1]{\boldsymbol{A}_{#1}} 
 \newcommand{\fbold}[1]{\boldsymbol{f}_{#1}} 
 
\begin{tiny}
\author{
\small{Paola F. Antonietti
\thanks{MOX, Dipartimento di Matematica, Politecnico di Milano, E-mail: {\tt paola.antonietti@polimi.it}},
Lorenzo Mascotto
\thanks{Dip. di Matematica,  Universit\`a degli Studi di Milano, E-mail: {\tt lorenzo.mascotto@unimi.it}
and Inst. f\"ur Mathematik, C. von Ossietzky Universit\"at Oldenburg, E-mail: {\tt mascotto.lorenzo@uni-oldenburg.de}},
Marco Verani}
\thanks{MOX, Dipartimento di Matematica, Politecnico di Milano, E-mail: {\tt marco.verani@polimi.it}}}
\end{tiny}

\date{}

\title{\textbf{\normalsize{A multigrid algorithm for the $\p$--version of the Virtual Element Method}}
\thanks{Paola F. Antonietti has been partially supported by the Fondazione Cariplo and Regione Lombardia research grant n. 2015-0182: ``PolyNum: Polyhedral numerical methods for partial differential equations''. 
M. Verani has been partially supported by the Italian research grant  {\sl Prin 2012}  2012HBLYE4  ``Metodologie innovative nella modellistica differenziale numerica'' and by INdAM-GNCS}}

\begin{document}
\maketitle

\begin{abstract}
We present a multigrid algorithm for the solution of the linear systems of equations stemming from the $\p-$version of the Virtual Element discretization of a two-dimensional Poisson problem. The sequence of coarse spaces are constructed decreasing progressively the polynomial approximation degree of the Virtual Element space, as in standard $p$-multigrid schemes.
The construction of the interspace operators relies on auxiliary Virtual Element spaces, where it is possible to compute higher order polynomial projectors.
We prove that the multigrid scheme is uniformly convergent, provided the number of smoothing steps is chosen sufficiently large. We also demonstrate that the resulting scheme provides a uniform preconditioner with respect to the number of degrees of freedom that can be employed to accelerate the convergence of classical Krylov-based iterative schemes.  Numerical experiments validate the theoretical results. 
\end{abstract}

\section*{Introduction} \label{section introduction}
 In recent years there has been a tremendous interest in developing numerical methods for the approximation of partial differential equations where the finite-dimensional space is built upon an underlying mesh composed by arbitrarily-shaped polygonal/polyhedral (polytopic, for short) elements. 
Examples of methods that have been proposed so 
far include Mimetic Finite Differences \cite{HymanShashkovSteinberg_1997,BrezziLipnikovShashkov_2005,Lipnikov-Manzini-Shashkov:2014,BeiraoManziniLipnikov_2014},  
Polygonal Finite Element Methods \cite{SukumarTabarraei_2004,SukumarTabarraei_2008}, 
Discontinuous Galerkin Finite Element Methods \cite{AntBreMar09,dg_cfes_2012,BasBotColSub,Canetal14,CangianiDongGeorgoulisHouston_2016,AntoniettiFacciolaRussoVerani_2016,CangianiDongGeorgoulis_2016}
Hybridizable and Hybrid High Order Methods \cite{CockburnDongGuzman_2008,DiPietroErnLemaire_2014,CockburnDiPietroErn_2016}, and
Gradient schemes \cite{EymardGuichardHerbin_2012,DroniouEymardGallouetHerbin_2013},
for example.   Recently, in \cite{VEMvolley} the Virtual Element Method (VEM) has been introduced, and further developed for elliptic and parabolic problems in \cite{BeiraoBrezziMariniRussso_2016b, VaccaBeirao_2014}.
VEMs for linear and nonlinear elasticity have been developed in \cite{VEMelasticity,GainTalischiPaulino_2014,BeiraoLovadinaMora_2015}, 
whereas VEMs for plate bending, Cahn-Hilliard, Stokes, and Helmholtz problems have been addressed in \cite{BrezziMarini_2013, AntoniettiBeiraoScacchiVerani_2016,AntoniettiBeiraoMoraVerani_2014,PerugiaPietraRusso_2016}.
For discrete topology optimization and fracture networks problems we refer to \cite{AntoniettiBruggiScacchiVerani_2016_b} and \cite{BenedettoBerronePieracciniScialo_2014}, respectively.
Moreover, several variants of the Virtual Element Method, including mixed, $H(\text{div})$ and $H(\bold{curl})$-conforming,
serendipity and nonconforming VEMs have been proposed in \cite{BrezziFalkMarini_2014,BeiraoBrezziMariniRusso_2016,BeiraoBrezziMariniRusso_2016b,BeiraoBrezziMariniRusso_2016c,AyusoLipnikovManzini_2016,
CangianiGyryaManzini_2016,ZhaoChenZhang_2016,AntoniettiManziniVerani_2016}.

All the above mentioned contributions focus on the $h$--version of the Virtual Element Method.
The $\p$-version VEM was introduced in \cite{hpVEMbasic} for the 2D Poisson problem, considering quasi-uniform meshes.
It was shown that, analogously to the $\p$-version Finite Element Method (FEM) case, 
if the solution of the problem has fixed Sobolev regularity, then the convergence rate of the method in terms of $\p$ is algebraic,
whereas if the solution is analytic then the convergence rate is $\p$ exponential.
In \cite{preprint_hpVEMcorner}, the full $h\p$-version VEM was studied based on employing meshes geometrically graded towards the corners of the domain and properly choosing the
distribution of polynomial approximation degree, so that the convergence rate of the method is exponential in terms of the number of degrees of freedom.\\

So far, the issue of developing efficient solution techniques for the linear systems of equations stemming 
from both the $h$-, $p$- or $hp$-versions of the VEM
has not been addressed yet.  The main difficulty in the
development of optimal (multilevel) solution techniques relies on the
construction of consistent coarse solvers which are non-trivial
on grids formed by general polyhedra.  Very recently, using the techniques of \cite{AntoniettiSartiVerani_2015,AntoniettiSartiVerani_DD22}
a multigrid algorithm for the $hp$-version Discontinuous Galerkin methods
on agglomerated polygonal/polyhedral meshes has been analyzed in
\cite{AXHSV_calcolo}.
\\

The aim of this paper is to develop efficient iterative solvers for the solution of the linear systems of equations stemming from the $\p$--version of the Virtual Element discretization of a two-dimensional Poisson problem.
We propose to employ a W-cycle $\p$-multigrid multigrid algorithm, i.e. coarse levels are obtained by decreasing progressively the polynomial approximation degree up to the coarsest level which corresponds to the lowest (linear) Virtual Element (VE in short) space.
The key point is the construction of suitable prolongation operators between the hierarchy of VE spaces.
With the standard VE space such prolongation operators cannot be constructed based on employing
only the degrees of freedom. For such a reason we introduce a suitable auxiliary VE space, which is identical to the standard VE space from the algebraic point of view and which allows to construct computable interspace operators but results into non-inherited sublevel solvers.
This in turn complicates the analysis of the multigrid algorithms, since we need to account for non-inherited sublevel solvers. 
Employing a Richardson smoother and following the classical framework, see e.g. \cite{BrennerScott}, we prove that the W- cycle algorithm converges uniformly provided the number of smoothing steps is chosen sufficiently large.  
We also demonstrate that the resulting multigrid algorithm provides a uniform preconditioner for the Preconditioned Conjugate Gradient method (PCG),
i.e., the number of PCG iterations needed to reduce the (relative) residual up to a (user-defined) tolerance is uniformly bounded independently 
of the number of degrees of freedom. Further, employing the Gauss-Seidel smoother in place of the Richardson one can improve the performance of our iterative scheme.

The extension of the present setting to $h$-multigrid methods, i.e. where the coarse levels are formed by geometric agglomeration of the underlying grid is currently under investigation.\\

The remaining part of the paper is organized as follows. In Section~\ref{section model problem and a virtual element method}, we introduce the model problem, a Virtual Element Method approximating its solution, the associated linear system and the multigrid algorithm;
moreover, an auxiliary VE space, needed for the construction of the algorithm, is presented.
In Section~\ref{section multigrid methods with non inherited sublevel solvers}, we present the $W$-cycle $p$ VEM multigrid algorithm; its convergence analysis is the topic of Section~\ref{section convergence analysis of the multigrid method}.
Finally, in Section~\ref{section numerical result}, numerical results are shown.\\

Throughout the paper, we will adopt the standard notation for Sobolev spaces (see \cite{evansPDE, adamsfournier}).
In particular, given $\omega \subset \mathbb R^2$, $L^2(\omega)$ and $H^1(\omega)$ are the standard Lebesgue and Sobolev spaces over $\omega$, respectively, and 
$\Vert \cdot \Vert_{0,\omega}$, $\Vert \cdot \Vert _{1,\omega}$, and $\vert \cdot \vert _{1,\omega}$,
are the Lebesgue and the Sobolev (semi)norms, respectively. We will write $x\lesssim y$ and $x \approx y$ meaning that there exist positive constants $\c_1$, $\c_2$ and $\c_3$ independent of the discretization and multigrid parameters, such that $x\le \c_1 y$ and $\c_2 y \le x \le \c_3 y$, respectively. In addiction, $\mathbb P _\ell(\omega)$, $\omega  \subset \mathbb R^d$, $d=1,2$, denotes the space of polynomials of maximum degree $\ell \in \mathbb N$ over $\omega$, with the the convention $\mathbb P_{-1}(\omega)= \emptyset$.
We will also employ the standard multi-index notation:
\begin{equation} \label{multiindex notation}
\mathbf v = (v_1,v_2),\; \boldalpha = (\alpha_1, \alpha_2),\quad \mathbf v^{\boldalpha} = v_1^{\alpha_1} v_2 ^{\alpha_2}.
\end{equation}

\section {The model problem and the $\p$-version Virtual Element Method} \label{section model problem and a virtual element method}
Let $\Omega \subset \mathbb R^2$ be a polygonal domain and $\f \in L^2(\Omega)$ we consider the following model problem: find $\u \in \V = H^1_0(\Omega) \text{ such that}$:
\begin{equation} \label{continuous weak formulation}
\a(\u,\v) = \int _\Omega \f \v,\quad \forall \v \in \V,
\end{equation}
$\text{where } \a(\cdot, \cdot) = (\nabla \cdot, \nabla \cdot)_{0,\Omega}$.
Problem \eqref{continuous weak formulation} is well-posed, cf. \cite{BrennerScott}, for example.
In the next section we introduce the $\p$-version of the Virtual Element Method and we discuss its implementation.
In Section~\ref{subsection an auxiliary virtual element space}, we build an \emph{auxiliary} VE space that will be instrumental to construct and analyse our multigrid algorithm.
\subsection{The $\p$-version Virtual Element Method} \label{subsection the p version of the virtual element method}
In this section, we introduce the $\p$--version Virtual Element Method, based on polygonal meshes with straight edges for the discretization of problem \eqref{continuous weak formulation}.

Let $\tauh$ be a \emph{fixed} decomposition of $\Omega$ into non-overlapping polygonal elements $\E$, and let $\Nu$ and $\mathcal E$ be the set of all vertices and edges of $\tauh$, respectively. We set $\Nu_b= \Nu \cap \partial \Omega$ and $\mathcal E_b=\mathcal E \cap \partial \Omega$.
Given $\E$ generic polygon in $\tauh$, we also define $\Nu^\E = \Nu \cap \partial \E$ and $\mathcal E^\E = \mathcal E \cap \partial \E$ as the set of vertices and edges of polygon $\E$, respectively.
To each edge $\e \in \mathcal E$, we associate $\boldsymbol \tau $  and $\n$, the tangential and normal unit vector (obtained by a counter-clockwise rotation of $\boldsymbol \tau$), respectively.\\

For future use, it is convenient to split the (continuous) bilinear form $\a(\cdot, \cdot)$ defined in \eqref{continuous weak formulation} into a sum of local contributions:
\[
\a(\u,\v) = \sum_{\E \in \tauh} \aE(\u,\v)\quad \forall \u,\,\v\in \V,\quad\text{where } \aE(\cdot,\cdot) = (\nabla \cdot, \nabla \cdot)_{0,\E}.
\]
In order to construct the $\p$-VEM approximation of \eqref{continuous weak formulation}, we need the following ingredients:
\begin{itemize}
\item Finite dimensional subspaces $\VptildeE$ of $\VE=\V \cap H^1(\E)$ $\forall \E \in \tauh$ and a finite dimensional subspace $\Vp$ of $\V$, such that $\VptildeE=\Vptilde|_\E$;
\item Local symmetric bilinear forms $\apE: \VptildeE\times \VptildeE \rightarrow \mathbb R$ $\forall \E \in \tauh$ so that:
\begin{equation} \label{global bilinear form}
\ap(\up, \vp) = \sum_{\E \in \tauh} \apE(\up, \vp) \quad \forall \up,\,\vp \in \Vptilde;
\end{equation}
\item A duality pairing $\langle \fp, \cdot \rangle_\p$, where $\fp\in \Vptilde'$ and $\Vptilde'$ is the dual space of $\Vptilde$.
\end{itemize}
The above ingredients must be built  in such a way that the discrete version of \eqref{continuous weak formulation}:
\begin{equation} \label{discrete weak formulation}
\begin{cases}
\text{find } \up \in \Vptilde \text{ such that}\\
\ap(\up,\vp)=\langle \fp, \vp \rangle_\p,\quad \forall \vp \in \Vptilde,
\end{cases}
\end{equation}
is well-posed and optimal a priori energy error estimates hold, cf. \cite{hpVEMbasic}.\\

We begin by introducing the local space $\VptildeE$; given $\E\in\tauh$ and $\p \ge1$, we set:
\begin{equation} \label{classical choice virtual element}
\VptildeE = \left\{ \vp \in H^1(\E) \mid \Delta \vp \in \mathbb P_{\p-2}(\E),\, \vp|_{\partial \E} \in \mathbb B_\p(\partial \E)    \right\},
\end{equation}
where
\begin{equation} \label{boundary space}
\mathbb B _\p (\partial \E) = \left\{ \vp \in \mathcal C^0(\partial \E) \mid \vp|_\e \in \mathbb P_\p(\e),\, \forall \e \in \mathcal \E^\E  \right\}.
\end{equation}
We remark that the above definition coincides with the definition of the two dimensional VE space introduced in \cite{VEMvolley} for the Poisson equation, and that clearly $\mathbb P_\p(\E) \subseteq \VptildeE$, $\p\geq1$. The global space is then 
obtained by gluing continuously the local spaces:
\begin{equation} \label{global classical virtual space}
\Vptilde =\left\{ \vp \in H^1_0(\Omega) \cap \mathcal C^0(\overline \Omega) \mid \vp|_\E \in \VptildeE,\, \forall \E \in \tauh  \right\}.
\end{equation}
We note that for the sake of simplicity we are assuming a uniform $\p$ on each $\E \in \tauh$.
Nevertheless, it is possible to construct VEM with non-uniform degrees of accuracy over $\tauh$, see \cite{preprint_hpVEMcorner}.

We endow the space \eqref{classical choice virtual element} with the following set of degrees of freedom (dofs).
To every $\vp \in \VptildeE$ we associate:
\begin{itemize}
\item the values of $\vp$ at the vertices of $\E$;
\item the values of $\vp$ at $\p-1$ distinct internal nodes on each edge $\e \in \mathcal E^\E$;
\item the scaled internal moments:
\begin{equation} \label{internal moments}
\frac{1}{\vert \E \vert} \int _{\E} \vp\,\m_{\boldalpha},
\end{equation}
where $\m_{\boldalpha}$ is an $L^2(\E)$ orthonormal basis of $\mathbb P_{\p-2}(\E)$.
\end{itemize}
Reasoning as in \cite[Proposition 4.1]{VEMvolley}, it is easy to see that this is a unisolvent set of degrees of freedom.
We observe that the basis $\{\m_{\boldalpha} \}_{\vert \boldalpha \vert =0}^{\p-2}$ introduced in \eqref{internal moments} can be built by orthonormalizing (following, e.g., \cite{BassiBottiColomboDipietroTesini}) for instance the monomial basis given by:
\begin{equation} \label{monomial basis}
\quad q_{\boldalpha} = \left( \frac{\mathbf x - \mathbf x_\E}{\diam(\E)} \right)^{\boldalpha}\quad \forall \boldalpha \in \mathbb N^2, \; \vert \boldalpha \vert \le \p-2,
\end{equation}
where $\mathbf x_\E$ denotes the barycenter of the element $\E$.
\begin{remark}
The definition of the internal degrees of freedom in \eqref{internal moments} differs from its classical counterpart in \cite{VEMvolley, hpVEMbasic} where the internal moments are defined with respect to the monomial basis \eqref{monomial basis}.
The new choice of the internal degrees of freedom will play a crucial role in the choice of the stabilization of the method, see Theorem \ref{theorem stability dof dof product}, and in the choice of the space-dependent inner product associated with the multigrid algorithm, see Theorem \ref{theorem stability bounds space dependent inner product} below.
\end{remark}
We define the canonical basis $\{\varphi_j\}_{j=1}^{\dim(\VptildeE)}$ as:
\begin{equation} \label{canonical basis}
\dof_i (\varphi_j) =\delta_{ij},\; i,j=1,\dots, \dim(\VptildeE), \quad \text{where } \delta_{ij} \text{ is the Kronecker delta}.
\end{equation}
Owing the definition \eqref{classical choice virtual element} of the local VE space and the choice of the degrees of freedom, it is possible to compute the following operators:
\begin{itemize}
\item the $L^2(\E)$ projection operator $\Pizpmd : \VptildeE\rightarrow \mathbb P_{\p-2}(\E)$:
\begin{equation} \label{classical L2 orthogonality}
(\Pizpmd\vp -\vp, \q)_{0,\E} = 0 \quad 	\forall \vp  \in \VptildeE,\; \forall \q \in \mathbb P_{\p-2}(\E);
\end{equation}
\item the $H^1(\E)$ projector $\Pinablap: \VptildeE \rightarrow \mathbb P_\p(\E)$:
\begin{equation} \label{classical H1 orthogonality}
\begin{cases}
\aE(\Pinablap \vp-\vp, \q)=0, & \forall\q \in \mathbb P_\p(\E),\\
\int_\E(\Pinablap \vp -\vp)=0, & \text{ if } \p\ge 2,\\
\int_{\partial \E}(\Pinablap \vp - \vp)=0, & \text{ if } \p = 1,\\
\end{cases}
\quad \quad \quad \forall \vp \in \VptildeE,
\end{equation}
\end{itemize}
see \cite{VEMvolley, hitchhikersguideVEM} for details.
We observe that the last two conditions in \eqref{classical H1 orthogonality} are needed in order to fix the constant part of the energy projector.\\

Next, we introduce the discrete right-hand side $\fp \in \Vptilde^{\prime}$ and the associated duality pairing:
\begin{equation} \label{discrete loading term}
\langle \fp, \vp \rangle_\p = \sum _{\E \in \tauh} \int_{\E} \Pi^0_{\max(\p-2,1)}\f  {\vp},
\end{equation}
where
\[
\overline {\v}_\p = \begin{cases}
\frac{1}{\vert \partial \E \vert} \int_{\partial \E} \vp 	& \text{if } \p=1,\\
\vp								& \text{if } \p \ge 2.\\
\end{cases}
\]
We observe that it is possible to compute up to machine precision the expression in \eqref{discrete loading term},
because the action of the projector $\Pi^0_{\max(\p-2,1)}$ on all the elements of $\VpE$ is computable.
For a deeper study concerning the approximation of the discrete loading term see \cite{AhmadAlsaediBrezziMariniRusso_2013, VEMelasticity, hpVEMbasic}.\\

Finally, we turn our attention to the local and global discrete bilinear forms.
We require that the local bilinear forms $\apE: \VpE \times \VpE \rightarrow \mathbb R$ satisfy, for all $\E \in \tauh$, the two following assumptions.
\begin{itemize}
\item[(\textbf{A1})] \textbf{$\p$ consistency}:
\begin{equation} \label{p consistency}
\aE(\q,\vp) = \apE(\q,\vp) \quad \forall \q \in \mathbb P_{\p}(\E),\,\forall \vp \in \VptildeE;
\end{equation}
\item[(\textbf{A2})] \textbf{stability}: there exist two positive constants $0< \alpha_*(\p)<\alpha_*(\p)< +\infty$, possibly depending on $\p$, such that:
\begin{equation} \label{stability formula}
\alpha_*(\p)\vert \vp \vert_{1,\E}^2 \le \apE(\vp, \vp) \le \alpha^*(\p) \vert \vp \vert _{1,\E} ^2  \quad \forall \vp \in \VptildeE.
\end{equation}
\end{itemize}
Assumption (\textbf{A1}) guarantees that the method is exact whenever the solution of \eqref{continuous weak formulation} is a polynomial of degree $\p$, whereas assumption (\textbf{A2}) guarantees the well-posedness of problem \eqref{discrete weak formulation}.
Let now $\Id_\p$ be the identity operator on the space $\VptildeE$, we set:
\begin{equation} \label{local discrete bilinear form}
\apE(\up, \vp) = \aE(\Pinablap \up, \Pinablap \vp) +  \SEp((\Id_\p - \Pinablap)\up, (\Id_\p-\Pinablap)\vp) \quad \forall \up,\vp \in \VptildeE,
\end{equation}
where $\Pinablap$ is defined in \eqref{classical H1 orthogonality} and the local bilinear form $\SEp(\cdot, \cdot)$ as:
\begin{equation} \label{first stabilization}
\SEp (\up, \vp) = \sum_{i=1}^{\dim(\VpE)} \dof_i(\up) \dof_i(\vp)
\end{equation}
satisfies:
\begin{equation} \label{stabilization abstract}
\c_*(\p) \vert \vp \vert_{1,\E}^2 \le \SEp(\vp, \vp)  \le \c^*(\p) \vert \vp \vert_{1,\E}^2\quad \forall \vp \in \ker(\Pinablap),
\end{equation}
where $\c_*(\p)$ and $\c^*(\p)$ might depend on $\p$. We underline that the local discrete bilinear form \eqref{local discrete bilinear form} satisfies (\textbf{A1}) and (\textbf{A2}) and, thanks to \eqref{stabilization abstract}, the 
following bounds hold:
\begin{equation*}
\begin{aligned}
& \alpha_*(\p) \vert \up \vert_{1,\Omega}^2 \lesssim  \ap(\up, \up),
&& \ap(\up, \vp) \lesssim \alpha^*(\p) \vert \up \vert_{1,\Omega} \vert \vp \vert_{1,\Omega} 
&& \forall \up,\vp \in \Vptilde,\\ 
\end{aligned}
\end{equation*}
with:
\begin{equation} \label{alpha and c}
\alpha_*(\p) = \min (1,c_*(\p)),\qquad \alpha^*(\p)= \max(1,c^*(\p)).
\end{equation}
The following result provides bounds in terms of $\p$ for the constants $c_*(\p)$ and $c^*(\p)$ in \eqref{stabilization abstract}.
\begin{thm} \label{theorem stability dof dof product}
Let $\E \in \tauh$ and let $\SEp(\cdot, \cdot)$ be the stabilizing bilinear form defined in \eqref{first stabilization}. Then
\[
\c_*(\p) \vert \vp \vert_{1,\E}^2 \lesssim \SEp(\vp,\vp) \lesssim c^*(\p) \vert \vp \vert_{1,\E}^2 \quad \forall \vp \in \ker (\Pinablap),
\]
where $c_*(\p) \gtrsim \p^{-6}$  and $c^*(\p) \lesssim \p^4$.
\end{thm}
\begin{proof}
The thesis follows by combining the forthcoming technical Lemmata \ref{lemma stabilization corner singularities} and \ref{lemma link between the two stabilizations}.
\end{proof}
An immediate consequence of Theorem \ref{theorem stability dof dof product} and \eqref{alpha and c} is that  it holds:
\begin{equation} \label{bounds stability constants}
\alpha_*(\p) \gtrsim \p^{-6}, \quad \alpha^*(\p) \lesssim
\p^4,
\end{equation}
where $\alpha_*(\p)$ and $\alpha^*(\p)$ are given in \eqref{stability formula}.
\begin{lem} \label{lemma stabilization corner singularities}
Let $\StildeEp(\cdot, \cdot)$ be the local \emph{auxiliary} stabilization defined as:
\begin{equation} \label{corner sing stab}
\StildeEp (\up,\vp) = \frac{p}{h_\E} (\up, \vp)_{0,\partial \E} + \frac{p^2}{h_\E^2} (\Pizpmd \up, \Pizpmd \vp)_{0,\E}.
\end{equation}
Then, it holds:
\begin{equation} \label{stability bounds corner singularities sing stab}
\begin{split}
&\c_*(\p) \vert \vp \vert_{1,\E}^2 \lesssim \StildeEp (\vp,\vp) \lesssim c^*(\p) \vert \vp \vert_{1,\E}^2 \quad \quad \forall \vp \in \ker (\Pinabla_\p),
\end{split}
\end{equation}
where $c_*(\p) \gtrsim\p^{-5}$, $\c^*(\p) \lesssim \p^2$, and where $\Pinabla_p$ is the energy projector defined in \eqref{classical H1 orthogonality}.
\end{lem}
\begin{proof}
The thesis follows based on employing integration by parts,  the properties of orthogonal projections and $hp$ polynomial inverse estimates. It follows the lines of the proof of \cite[Theorem 4.1]{preprint_hpVEMcorner}; for the sake of brevity the details are not reported here.
\end{proof}
\begin{lem} \label{lemma link between the two stabilizations}
Let $\E \in \tauh$ and let $\SEp$ and $\StildeEp$ be defined as in \eqref{first stabilization} and \eqref{corner sing stab}, respectively. Then, 
\[
\p^{-1} \SEp(\vp,\vp) \le \StildeEp(\vp, \vp) \lesssim \p^2 \SEp (\vp, \vp) \quad \forall \vp \in \VptildeE.
\]
\end{lem}
Before showing the proof, we recall that given
$\{\rho_j^{\p+1}\}_{j=0}^\p$ and $\{\xi _j\}_{j=0}^{\p}$ the $\p+1$ Gau\ss-Lobatto nodes and weights on $\Ihat = [-1,1]$, respectively, 
it holds:
\begin{equation} \label{Bernardi Maday stabilization}
\sum_{j=0}^\p \q^2(\xi_j^{\p+1}) \rho_j^{\p+1} \lesssim \Vert \q \Vert_{0,\Ihat}^2 \le \sum_{j=0}^\p \q^2(\xi_j^{\p+1}) \rho_j^{\p+1}\quad \forall \q \in \mathbb P_\p(\Ihat),
\end{equation}
cf. \cite[(2.14)]{bernardimaday1992polynomialinterpolationinsobolev}. Moreover, it holds:
\begin{equation} \label{bounds on Gauss Lobatto weights}
\p^{-2}\lesssim \rho_j ^{\p+1} \lesssim 1 \quad \forall j=0,\dots, \p+1,
\end{equation}
where the hidden constants are positive and independent of $\p$, see \cite[(25.4.32)]{AbramowitzStegun_handbook}.
\begin{proof}
By using \eqref{Bernardi Maday stabilization} and \eqref{bounds on Gauss Lobatto weights}, we obtain:
\begin{equation} \label{equivalence estimate boundary term}
\frac{1}{h_\E}\p^{-1} \sum_{j=1}^{\card(\mathcal E^\E)} \dof_{b,j}^2(\vp) \lesssim \frac{\p}{h_\E} \Vert \vp \Vert_{0,\partial \E}^2 \lesssim \frac{\p}{h_\E} \sum_{j=1}^{\card(\mathcal E ^\E)} \dof_{b,j}^2(\vp),
\end{equation}
where $\dof_{b,j}$ denotes the $j$-th boundary degree of freedom. This concludes the discussion concerning the boundary term.
Next, we study the bulk term in \eqref{corner sing stab}, and consider the expansion of $\Pizpmd \vp$ into the $L^2(\E)$ orthonormal basis $\{\m_{\boldalpha}\}_{\vert \boldalpha \vert = 0}^{\p-2}$ introduced in \eqref{internal moments}:
\begin{equation} \label{expansion PIZPM1}
\Pizpmd \vp = \sum_{\vert \boldalpha \vert \le \p-2} c_{\boldalpha} \m_{\boldalpha}.
\end{equation}
Testing \eqref{expansion PIZPM1} with $\m_{\boldbeta}$, $\vert \boldbeta \vert \le \p-2$, we obtain:
\[
\vert \E \vert \dof_{\boldbeta} (\vp) = \int_\E \vp \m_{\boldbeta} = \int_\E \Pizpmd \vp \m_{\boldbeta} = c_{\boldbeta},
\]
where $\dof_{\boldbeta}(\cdot)$ denotes the internal degrees of freedom associated with polynomial $\m_{\boldbeta}$.
 As a consequence:
\begin{equation} \label{L2 projectors are equal}
\Pizpmd \vp =  \sum_{\vert \boldalpha \vert \le \p-2} \vert \E \vert \dof_{\boldalpha}(\vp) \m_{\boldalpha}.
\end{equation}
Parceval identity implies:
\begin{equation} \label{equivalence estimate internal term}
\frac{\p^2}{h_\E^2} (\Pizpmd \vp, \Pizpmd \vp) _{0,\E} = \frac{\p^2}{h_\E^2} \sum_{\vert \boldalpha \vert \le \p-2} \vert \E \vert ^2 \dof_{i,\boldalpha}^2(\vp),
\end{equation}
where $\dof_{i,\vert \boldalpha \vert}(\cdot)$ denotes the internal degree of freedom associated with polynomial $\m_{\boldalpha}$.
The thesis follows from \eqref{equivalence estimate boundary term} and \eqref{equivalence estimate internal term}.
\end{proof}
\begin{remark}
In order to guarantee the proper scaling in terms of $h$ in \eqref{first stabilization},
we should multiply the internal dofs with $\vert \E \vert$, see \eqref{equivalence estimate internal term}, and the boundary dofs by $h_\E^{-1}$, see \eqref{equivalence estimate boundary term}.
Since we consider only the $\p$--version of the virtual element method, then we can drop these scaling factors.
\end{remark}

Finally, from \cite{VEMvolley, hpVEMbasic} the following error bound in the energy norm holds:
\begin{equation} \label{three terms error decomposition}
\vert \u - \up \vert_{1,\Omega} \lesssim \frac{\alpha^*(\p)}{\alpha_*(\p)} \left\{ \mathcal F _\p  + \inf_{\uI \in \Vp} \vert \u - \uI \vert_{1,\Omega} + \sum_{\E \in \tauh} \inf_{\upi \in \mathbb P_{\p}(\E)} \vert \u -\upi \vert_{1,\E}  \right\},
\end{equation}
where $\u$ and $\up$ are the solution of \eqref{continuous weak formulation} and \eqref{discrete weak formulation}, respectively, 
$\alpha_*(\p)$ and $\alpha^*(\p)$ are the stability constants given in \eqref{stability formula} and $\mathcal F_\p$ is the smallest constant satisfying:
\[
(\f,\vp)_{0,\Omega} - \langle \fp , \vp \rangle_p \le \mathcal F_\p \vert \vp \vert_{1,\Omega} \quad \forall \vp \in \Vp.
\]
From \eqref{three terms error decomposition} and following \cite{hpVEMbasic}, it is possible to prove $\p$ error bounds analogous to those in the $\p$-FEM case, see \cite{SchwabpandhpFEM}.\\

Finally, we focus on the construction of the linear system of equations stemming from \eqref{discrete weak formulation}.
By expanding the trial function $\up$ as a combination of the elements in the canonical basis, see \eqref{canonical basis}, 
\[
\up = \sum_{i=1}^{\dim(\Vptilde)} \dof_i(\up) \varphi_i = \sum_{i=1}^{\dim(\Vptilde)} (\upvect)_i \varphi_i,
\]
where $\upvect \in \mathbb R^{\dim(\Vptilde)}$ is the set of dofs of $\up$, and selecting $\vp$ as $\varphi_j$, $j=1,\dots, \dim(\Vptilde)$, we obtain:
\begin{equation} \label{system arising from VEM}
\Kpbold \cdot \upvect = \fpvect,
\end{equation}
where:
\begin{equation} \label{definition stiffness matrix and discrete rhs}
\begin{aligned}
&(\Kpbold)_{i,j} = \ap(\varphi_j, \varphi_i),
& (\fpvect)_i = \langle \fp, \varphi_i \rangle _\p,
&& i,\,j=1,\dots, \dim(\Vptilde),
\end{aligned}
\end{equation}
Both the right-hand side and the coefficient matrix are computable exactly up to machine precision, see \cite{hitchhikersguideVEM}.
In the next section, we discuss the spectral condition number of the stiffness matrix $\Kpbold$.
\subsubsection{The condition  number of the stiffness matrix $\Kpbold$} \label{subsection the condition number of the stiffness matrix Kp}
In \eqref{internal moments} we defined the internal degrees of freedom associated with space $\Vp(\E)$ defined in \eqref{classical choice virtual element}
as the scaled moments with respect to an $L^2(\E)$ orthonormal basis of $\mathbb P_{\p-2}(\E)$.
We observe that this choice is different from the usual choice adopted in standard VEM literature, see e.g. \cite{VEMvolley, hitchhikersguideVEM, AhmadAlsaediBrezziMariniRusso_2013},
where the internal dofs are defined as the (scaled) moments with respect to the monomial basis \eqref{monomial basis} of $\mathbb P_{\p-2}(\E)$.
Our choice, which is a key ingredient in the proof of Lemma \ref{lemma link between the two stabilizations}. also plays a fundamental role in the spectral properties of the stiffness matrix $\Kpbold$ defined in \eqref{definition stiffness matrix and discrete rhs}.
Indeed, in Table \ref{figure comparison condition number} we compare the spectral condition number $\kappa(\Kpbold)$
 of the stiffness matrix $\Kpbold$ as a function of the degree of accuracy of the method $\p$, based on employing the two different sets of internal degrees of freedom,
namely the scaled moments with respect to the $L^2(\E)$ orthonormal basis of $\mathbb P_{\p-2}(\E)$ or with respect to the monomial basis \eqref{monomial basis}.
Results reported in Table \ref{figure comparison condition number} have been obtained on the Voronoi-Lloyd polygonal mesh shown in Figure \ref{meshes employed};
the same kind of results have been obtained on meshes made of squares and of quasi-regular hexagons. For the sake of brevity these results have been omitted.
From  the results reported in Table \ref{figure comparison condition number} it is clear that $\Kpbold$ grows, as for classical finite element methods,
as $\p^4$ whenever the interior dofs are defined with respect to an $L^2(\E)$ orthonormal basis of $\mathbb P_{\p-2}(\E)$ whereas the condition number $\Kpbold$ blows up exponentially 
if the scaled moments are defined with respect to the monomial basis \eqref{monomial basis}.
That is, the choice \eqref{internal moments} for the internal degrees of freedom is the right choice as it damps the condition number of the stiffness matrix effectively and prevents round off errors,
as those observed, for example,  in \cite{hpVEMbasic} where the monomial basis \eqref{monomial basis} was employed.
\begin{table}[!htbp]
\centering
\begin{tabular}{r cl }
$\p$
&$\kappa(\Kpbold)$    
&$\kappa(\Kpbold)$ \\
\hline                     
& &  \\[-0.2cm]
$1$ & 1.3225e+01 & 7.2732e+01        \\ 
$2$ & 4.9712e+02 &1.0964e+03        \\ 
$3$ & 7.5099e+02 & 1.2910e+05    \\
$4$ & 1.1823e+03  & 1.5566e+07      \\ 
$5$ & 1.9395e+03  & 1.6003e+09    \\ 
$6$  & 3.5100e+03   & 1.7069e+11      \\ 
$7$ & 6.0754e+03 & 1.7280e+13   \\ 
$8$ & 1.0100e+04 & 1.7172e+15      \\ 
\hline
Rate & $p^4$ & $a\exp(bp)$, $a=0.18$, $b=4.59$
\end{tabular}
\caption{Condition number $\kappa(\Kpbold)$ of the stiffness matrix $\Kpbold$ as a function of  $\p$ for two different sets of internal degrees of freedom: 
\emph{(left)} scaled moments with respect to an $L^2(\E)$ orthonormal basis of $\mathbb P_{\p-2}(\E)$ (orthogonalized basis);
\emph{(right)}  scaled moments with respect to the monomial basis \eqref{monomial basis} of $\mathbb P_{\p-2}(\E)$ (monomial basis).
Voronoi-Lloyd polygonal mesh.} 
\label{figure comparison condition number}
\end{table}
\subsection{An auxiliary Virtual Element Space} \label{subsection an auxiliary virtual element space}
In this section, we introduce an auxiliary VE space which will be crucial for the construction of the multigrid algorithm in Section~\ref{section multigrid methods with non inherited sublevel solvers}.
Hence, following the spirit of \cite{AhmadAlsaediBrezziMariniRusso_2013}, we consider a modification of $\VptildeE$ into a \emph{diverse} space on which we are able to compute a higher order $L^2$ projector.
In particular, we set:
\begin{equation}\label{our choice virtual space}
\VpE = \left\{ \vp \in H^1(\E) \left \vert\vp|_{\partial \E} \in \mathbb B _\p(\partial \E),\, \Delta \vp \in \mathbb P_{\p-1}(\E),\, \int_\E(\Pizpmd \vp -\vp) \m_{\boldalpha}=0,\, \vert \boldalpha \vert=\p-1 \right.   \right\},
\end{equation}
where we recall that $\boldalpha \in \mathbb N^2$ is a multi-index.

Henceforth, we will denote by the expression \emph{enhancing constraints} the following set of constraints employed in the definition of $\VpE$:
\begin{equation} \label{enhancing constraints}
\int_\E (\Pizpmd \vp - \vp) \, \m_{\boldalpha} = 0,\quad \quad \vert \boldalpha \vert =\p-1,\quad \forall \vp \in \VpE.
\end{equation}

The definition of $\VpE$ is different from the one presented in \cite{AhmadAlsaediBrezziMariniRusso_2013}. 
Moreover, we observe that $\mathbb P_{\p}(\E) \nsubseteq \VpE$, but $\mathbb P_{\p-2}(\E) \subseteq \VpE$.
To be more precise, owing to the $L^2(\E)$ orthonormality of the $\m_{\boldalpha}$, it holds in fact:
\[
\mathbb P_{\p-2}(\E) \oplus \left( \mathbb P_\p(\E) / \mathbb P_ {\p-2}(\E)  \right)^{\perp_{\mathbb P_{\p-1}(\E)}} \subseteq \VpE,
\]
where $\left( \mathbb P_\p(\E) / \mathbb P_ {\p-2}(\E)  \right)^{\perp_{\mathbb P_{\p-1}(\E)}}$
denotes the space of polynomials of degree $\p$, not in the space of polynomials of degree $\p-2$, orthogonal to all $\m_{\boldalpha}$ with $\vert \boldalpha \vert =\p-1$.

We endow $\Vp$ with the same degrees of freedom of the space $\Vptilde$ introduced in \eqref{classical choice virtual element}.
Using the \emph{auxiliary} local virtual space $\Vp$ introduced in \eqref{our choice virtual space}, it is clear that we are able to compute the following operator:
\begin{itemize}
\item $\Pizpmu: \VptildeE \rightarrow \mathbb P_{\p-1}(\E)$, the $L^2$ projection onto the space of polynomials of degree $\p-1$, defined as in \eqref{classical L2 orthogonality}.
\end{itemize}
We stress that there is no chance to be able to compute explicitly $\Pizpmu$ as a map defined on $\VptildeE$, since the internal degrees of freedom are up to order $\p-2$,
whereas this is possible in the new space $\VpE$ we can do that since \eqref{enhancing constraints} allows to compute internal moments up to order $\p-1$.

The global \emph{auxiliary} VE space is obtained again by gluing continuously the local spaces as done in \eqref{global classical virtual space}:
\begin{equation} \label{global auxiliary virtual space}
\Vp=\left\{ \vp \in H^1_0(\Omega) \cap \mathcal C^0(\overline \Omega) \mid \vp|_\E \in \VpE,\, \forall \E \in \tauh  \right\}.
\end{equation}

The choice of the discrete bilinear form $\ap$ and of the right-hand side $\fp$ in \eqref{discrete weak formulation} are exactly the same as those in Section~\ref{subsection the p version of the virtual element method} for the space $\Vptilde$.

It is crucial to remark that the linear systems stemming from the use of $\Vptilde$ and $\Vp$ are the same.
In fact, it is clear from \eqref{local discrete bilinear form} that the construction of the local discrete bilinear forms depends uniquely on the choice of the set of the degrees of freedom (which we recall are the same for the two spaces)
and the energy projector $\Pinablap$ defined in \eqref{classical H1 orthogonality}, which is computed without the need of \eqref{enhancing constraints}.

Also the construction of the discrete right-hand side \eqref{discrete loading term} does not depend on the choice of the space since the $L^2$ projector $\Pi_{\max(1,\p-2)}$ defined in \eqref{classical L2 orthogonality}
is built using the internal degrees of freedom only, while the enhancing constraints \eqref{enhancing constraints} are neglected.

\begin{remark} \label{remark differences between virtual spaces}
The aforementioned equivalence between the two linear systems associated with spaces $\Vptilde$ and $\Vp$ is of great importance in order to design and analyse the multigrid algorithm in Section~\ref{section multigrid methods with non inherited sublevel solvers}.
However, $\Vptilde$ and $\Vp$ have significant differences.

The first issue we want to highlight is that the method associated with space $\Vp$ \eqref{global auxiliary virtual space} is not a ``good'' method from the point of view of the approximation property.
It is possible to show $\p$ approximation results on the first and the third term on the right hand side of \eqref{three terms error decomposition} following for instance \cite[Sections 4 and 5]{hpVEMbasic}.
The problematic term is the second one, i.e. the best error term with respect to functions in the virtual space. The approach used in \cite{hpVEMbasic}, which is the $\p$--version of \cite[Proposition 4.2]{MoraRiveraRodriguez_2015},
does not hold anymore in the enhanced version of VEM.
At the best of the authors knowledge, the $\p$ approximation of the ``best virtual'' error term in enhanced space is still an open problem.
On the other hand, the error analysis with space $\Vptilde$ is available in \cite{hpVEMbasic}.

The second issue we underline,
is that the space $\Vp$ \eqref{global auxiliary virtual space} is more suited for the construction of the multigrid algorithm than space $\Vptilde$ \eqref{global classical virtual space},
as will be clear from Section~\ref{section multigrid methods with non inherited sublevel solvers}.
\end{remark}

Let us summarize the strategy we will follow.
We consider a discretization of \eqref{continuous weak formulation} by means of the Virtual Element Method \eqref{discrete weak formulation} employing as an approximation space $\Vptilde$ defined in \eqref{global classical virtual space}.
The associated linear system \eqref{system arising from VEM} coincides algebraically with the one arising by employing the VE space $\Vp$ defined in \eqref{global auxiliary virtual space}.
For this reason, we can solve system \eqref{system arising from VEM} by means of a multigrid algorithm based on the sequences of spaces $\Vp$ defined in \eqref{global auxiliary virtual space}.

Having the vector of degrees of freedom $\upvect$, one can reconstruct functions in two different spaces: either in space $\Vptilde$ defined in \eqref{global classical virtual space}, or in space $\Vp$ defined in \eqref{global auxiliary virtual space}.
The discrete solution in the former space is the one to be taken into account, since it has the proper $\p$ approximation properties.

\section{Multigrid algorithm with non-inherited sublevel solvers} \label{section multigrid methods with non inherited sublevel solvers}
In this section, we present a $\p$-VEM multigrid algorithm and the key ingredients for its formulation.

In the construction of our multigrid algorithm, we will make use of two key ingredients. 
The first one are suitable (computable) interspace operators, i.e.  prolongation/restriction operators between two VE spaces.
These operators will be constructed by employing the properties of the following space-dependent inner product:
\begin{equation} \label{space dependent inner product}
(\wp, \vp)_\p = \sum_{i=1}^{\dim(\Vp)} \dof_i(\wp) \dof_i(\vp) \quad \forall \wp,\vp \in \Vp.
\end{equation}
The second ingredient is a suitable smoothing scheme $\Bp$, which aims at reducing the high frequency components of the error.\\

We aim at introducing a multigrid iterative method for the solution of the linear system in \eqref{system arising from VEM}, which we recall is given by:
\begin{equation} \label{multigrid system}
 \Kpbold  \cdot  \upvect = \fpvect,
\end{equation}
where the coefficient matrix $\Kpbold$ and the right-hand side $\fpvect$ are the matrix representations with respect to the their expansion in the canonical basis of space $\Vp$, defined in \eqref{global auxiliary virtual space}, of the operators:
\begin{equation} \label{construction multigrid matrix and rhs}
(\Ap \wp, \vp)_\p = \ap(\wp,\vp),\quad \quad (\emph{\emph f} _\p,\vp)_\p = \langle \fp , \vp \rangle _\p,\quad \forall \wp,\,\vp \in \Vp,
\end{equation}
cf. \eqref{local discrete bilinear form} and \eqref{discrete loading term}, respectively.

In order to introduce our $\p$-multigrid method, we consider a sequence of VE spaces given by
$\widetilde{V}_p, \widetilde{V}_{p-1}, \ldots, \widetilde{V}_1$, where the $\ell$-th level is given by $\widetilde{V}_{p-\ell}$, $\ell=0,\ldots, p-1$.
Let now consider the linear system of equations on level $\ell$: $\Kbold{\ell}  \cdot  \zbold_{\ell} = \gbold_{\ell}$.
We denote by $\texttt{MG}(\ell,\gbold_{\ell}, \zbold_{\ell}^{(0)}, \md)$ one iteration obtained by applying the $\ell$-th level iteration of our MG scheme 
to the above linear system, with initial guess $\zbold_{\ell}^{(0)}$ and using $\md$ post-smoothing steps, respectively. 
For $\ell = 1$, (coarsest level) the solution is computed with a direct method, that is
$\texttt{MG}(1,\gbold_{1},\zbold_{1}^{(0)}, \md)= \Kbold{1}^{-1}\gbold_{1},$
while for $\ell>1$ we adopt the recursive procedure described in Algorithm~\ref{table algorithm}. 
\begin{algorithm}[H]
\caption{$\ell$-th level of the $\p$-multigrid algorithm	}
\label{table algorithm}
\begin{algorithmic}
\State \underline{{\it  Coarse grid correction}}:
\State $\rbold_{\ell-1} = \Ibold_{\ell-1}^\ell \cdot (\gbold_\ell - \Kbold{\ell} \cdot \zbold_\ell ^{(0)})$; (restriction of the residual)
\State $\ebarbold_{\ell-1} = \texttt{MG}(\ell-1, \rbold_{\ell-1}, \mathbf 0_{\ell - 1}, \md)$; (approximation of the residual equation \dots)
\State $\ebold     _{\ell-1} = \texttt{MG}(\ell-1, \rbold_{\ell-1}, \ebarbold_{\ell - 1}, \md)$; (\dots $\Kbold{\p-1} \cdot \zbold_{\p-1} = \rbold_{\p-1}$)
\State $\zbold_\ell^{(1)} = \zbold_\ell^{(0)} + \Ibold_{\ell-1}^\ell \cdot \ebold_{\ell-1}$; (error correction step)
\State \underline{{\it  Post-smoothing}}:
\For{$i=2:\md+1$}
\State $\zbold_{\ell}^{(i)}= \zbold_{\ell}^{(i-1)} + \Bbold_\ell ^{-1} \cdot (\gbold_\ell - \Kbold{\ell} \cdot \zbold_\ell ^{(i-1)}) $;
\EndFor
\vspace{0.3cm}
\State $\texttt{MG}(\ell,\gbold_{\ell}, \zbold_{\ell}^{(0)}, \md)=\mathbf{z}_{\ell}^{(\md+1)}.$
\end{algorithmic}
\end{algorithm}
In presenting Algorithm \ref{table algorithm}, we used some objects that are not defined so far. In particular, $\Ibold _{\ell-1}^\ell$ denotes the matrix representation of the interspace operators defined in Section \ref{subsection interspace operator},
while $\Bbold_\p$ denotes the matrix representation of the smoothing operator defined in Section \ref{subsection smoothing scheme and spectral bounds}.

For a given, user defined tolerance \texttt{tol} and a given initial guess $\upvect^{(0)}$, 
the full $\p$-multigrid algorithm employed to solve \eqref{multigrid system} is summarized in Algorithm \ref{table full algorithm}; its analysis  
is presented in the forthcoming Section~\ref{section convergence analysis of the multigrid method}.
\begin{algorithm}[H]
\caption{$\p$-multigrid algorithm: $\uptildevect=\texttt{MG}(p,(\textbf{\emph{\emph f}} _\p, \uptildevect^{(0)}, \md).$}  
\label{table full algorithm}
\begin{algorithmic}
\State $\mathbf{r}_{\p}^{(0)}=\textbf{\emph{\emph f}} _\p - \mathbf{A}_{\p} \cdot \uptildevect^{(0)}$;
\While{ $\|\mathbf{r}_{\p}^{(i)}\| \leq \texttt{tol} \|\fbold{p} \|$} 
\State $ \uptildevect^{(i+1)}=\texttt{MG}(\p, \textbf{\emph{\emph f}} _\p,  \uptildevect^{(i)}, \md)$;
\State $\mathbf{r}_{\p}^{(i+1)}=\textbf{\emph{\emph f}} _\p - \mathbf{A}_{\p} \cdot \uptildevect^{(i+1)}$;
\State $i \longrightarrow i+1$;
 \EndWhile
\end{algorithmic}
\end{algorithm}
\begin{remark} \label{remark on level with huge differences}
As a byproduct, we underline that it is possible to employ multigrid algorithms where two ``adjacent'' levels, associated to spaces $\V_{\p_1}$ and $\V_{\p_2}$, respectively, satisfy $\vert \p_1-\p_2 \vert \ge 2$. In such cases, to build the interspace operators, it suffices to modify the definition \eqref{our choice virtual space} by using a ``larger'' enhancing technique and imposing that the laplacian of functions in the virtual space is a polynomial of higher degree, and then reduce the space with additional constraints on the $L^2$-projectors.
\end{remark}
\subsection{Space-dependent inner products} \label{subsection space dependent inner product}
The aim of this section is to prove the following result on the space-dependent inner product \eqref{space dependent inner product}, which will be useful for the forthcoming analysis.
\begin{thm} \label{theorem stability bounds space dependent inner product}
Let $(\cdot, \cdot)_\p$ be defined as in \eqref{space dependent inner product}. Then, the following holds true:
\begin{equation} \label{bounds space dependent inner product}
\beta_*(\p) \vert \vp \vert^2_{1,\Omega} \lesssim (\vp,\vp)_{\p} \lesssim \beta^*(\p) \vert \vp \vert_{1,\Omega}^2 \quad \forall \vp \in \Vp,
\end{equation}
where $\beta_*(\p) \gtrsim \p^{-8}$ and $\beta^*(\p) \lesssim 1$.
\end{thm}
In order to prove Theorem \ref{theorem stability bounds space dependent inner product}, it suffices to combine the forthcoming technical results.
The first one makes use of the following \emph{auxiliary} space-dependent inner product defined as:
\begin{equation} \label{auxiliary space dependent inner product}
(\up,\vp)_{\p,\aux} = \sum_{\E \in \tauh} (\up,\vp)_{\p,\aux;\E} \quad \forall \up,\vp \in \Vp,
\end{equation}
where the local contributions read:
\begin{equation} \label{local auxiliary space dependent inner product}
(\up,\vp)_{\p,\aux;\E} = h_\E^{-1} (\up,\vp)_{0,\partial \E} + h_{\E}^{-2} (\Pizpmu \up, \Pizpmu \vp)_{0,\E} \; \up,\vp \in \VpE,\; \forall \E \in \tauh.
\end{equation}
\begin{lem} \label{lem stability bounds space dependent inner product}
Let $(\cdot, \cdot)_{\p,\aux}$ be defined in \eqref{auxiliary space dependent inner product}. Then, it holds:
\begin{equation} \label{stability bounds space dependent inner product}
\widetilde\beta_*(\p) \vert \vp \vert_{1,\Omega}^2 \lesssim (\vp,\vp)_{\p,\aux} \lesssim \widetilde\beta^*(\p) \vert \vp \vert_{1,\Omega}^2 \quad \forall \vp \in \Vp,
\end{equation}
where $\widetilde \beta_*(\p) \gtrsim \p^{-6}$ and $\widetilde \beta^*(\p) \lesssim 1$.
\end{lem}
Before showing the proof, we recall that from \cite[Theorem 7.5]{preprint_hpVEMcorner} the following inverse-type holds  holds:
\begin{equation} \label{inverse estimate negative norm polynomials}
\begin{aligned}
& \Vert \q \Vert_{0,\E} \lesssim (\p+1)^2 \Vert \q \Vert_{-1,\E}
&& \q \in \mathbb P_{\p}(\E),
\end{aligned}
\end{equation}
where:
\begin{equation} \label{dual norm H01}
\Vert \cdot \Vert_{-1,\E} = \sup_{\Phi\in H^1_0(\E)\setminus \{0\}} \frac{(\cdot, \Phi)_{0,\E}}{\vert \Phi \vert_{1,\E}}.
\end{equation}
\begin{proof}{(Proof of Lemma \ref{lem stability bounds space dependent inner product})}
The proof is slightly different from the one for the stability bounds \eqref{stability bounds corner singularities sing stab};
in fact, here we work on the complete virtual space and not on $\ker(\Pinablap)$, being $\Pinablap$ defined in \eqref{classical H1 orthogonality}. 
In the following, we neglect the dependence on the size of the elements since we are assuming that the mesh is fixed; the general case 
follows from a scaling argument.
The upper bound follows from a trace inequality and the stability of orthogonal projection
\[
(\vp, \vp)_{\p,\aux;\E} = \Vert \vp \Vert_{0,\partial \E} ^2  + \Vert \Pizpmu \vp \Vert^2_{0,\E} \lesssim \Vert \vp \Vert_{1,\E} \quad \forall \vp \in \Vp,
\]
and summing up on all the mesh elements and applying the Poincar\`e inequality.
For the lower bound, by using an integration by parts and the definition of the local \emph{auxiliary} space \eqref{our choice virtual space}, we have:
\begin{equation} \label{first step beta low star}
\vert \vp \vert_{1,\E}^2 = \int_{\E} \nabla \vp \cdot \nabla \vp = \int_{\E} - \Delta \vp \Pizpmu \vp +\int_{\partial \E} \frac{\partial \vp}{\partial \n} \vp.
\end{equation}
Owing to \eqref{inverse estimate negative norm polynomials} and recalling that $\Delta \vp \in \mathbb P_{\p-1}(\E)$, we deduce:
\begin{equation} \label{laplacian estimate}
\Vert \Delta \vp \Vert_{0,\E} \lesssim \p^2 \Vert \Delta \vp \Vert_{-1,\E} = \p^2 \sup_{\Phi \in H^1_0(\E)\setminus \{0\}} \frac{(\Delta \vp, \Phi)_{0,\E}}{\vert \Phi \vert_{1,\E}} =
\p^2 \sup_{\Phi \in H^1_0(\E)\setminus \{0\}} \frac{(\nabla \vp, \nabla \Phi)_{0,\E}}{\vert \Phi \vert_{1,\E}} \lesssim \p^2 \vert \vp \vert_{1,\E}.
\end{equation}
We bound now the two terms appearing on the right-hand side of \eqref{first step beta low star}. 
Applying \eqref{laplacian estimate}, we have:
\begin{equation} \label{integral over E}
\int_{\E} \Delta \vp \Pizpmu \vp \le \Vert \Delta \vp \Vert_{0,\E} \Vert \Pizpmu \vp\Vert_{0,\E} \lesssim \p^2 \Vert \Pizpmu \vp \Vert_{0,\E} \vert \vp \vert_{1,\E}.
\end{equation}
Applying next a Neumann trace inequality \cite[Theorem A.33]{SchwabpandhpFEM}, a one dimensional $hp$ inverse inequality, the interpolation estimates \cite{Triebel, tartar} and \eqref{laplacian estimate}, we get:
\begin{equation} \label{integral over boundary E}
\int_{\partial \E} \frac{\partial \vp}{\partial \n} \vp \le \left \Vert \frac{\partial \vp}{\partial \n} \right \Vert_{-\frac{1}{2},\partial \E} \Vert \vp \Vert_{\frac{1}{2},\partial \E}
\lesssim (\vert \vp \vert_{1,\E} + \Vert \Delta \vp \Vert_{0,\E}) \p \Vert \vp \Vert_{0,\partial \E}\lesssim \p^3 \Vert \vp \Vert_{0,\partial \E} \vert \vp \vert_{1,\E}.
\end{equation}
Substituting \eqref{integral over E} and \eqref{integral over boundary E} in \eqref{first step beta low star}, we obtain:
\[
\vert \vp \vert_{1,\E} \lesssim \p^3 (\Vert \vp \Vert_{0,\partial \E} + \Vert \Pizpmu \vp \Vert_{0,\E}),
\]
whence
\[
\vert \vp \vert_{1,\E}^2 \lesssim \p^6 (\vp,\vp)_{\p,\aux;\E}.
\]
The thesis follows summing on all the elements.
\end{proof}
\begin{lem} \label{lemma link between the two space dependent inner products}
Let $(\cdot, \cdot)_{\p,\aux}$ and $(\cdot, \cdot)_{\p}$ be defined as in \eqref{auxiliary space dependent inner product} and \eqref{space dependent inner product}, respectively. Then it holds:
\begin{equation} \label{link between the two space dependent inner products}
\p^{-2} (\vp, \vp)_{\p,\aux} \lesssim (\vp,\vp)_{\p} \lesssim (\vp,\vp)_{\p,\aux} \quad \forall \vp \in \Vp.
\end{equation}
\end{lem}
\begin{proof}
The proof is a straightforward modification of the one of Lemma \ref{lemma link between the two stabilizations}.
\end{proof}
\begin{remark} \label{remark comparison with DG-FEM}
The choice \eqref{space dependent inner product} for the space-dependent inner product is crucial for the construction of the interspace operators, see Section~\ref{subsection interspace operator}.
Moreover, we point out that it coincides with the usual choice for the space-dependent inner product in the $h\p$ DG-Finite Element Framework, see \cite{AntoniettiSartiVerani_2015,AXHSV_calcolo}.
The Finite Element counterpart of Theorem \ref{theorem stability bounds space dependent inner product} is much less technical, since it suffices to choose an $L^2$ orthonormal basis of polynomials as canonical basis;
via Parceval identity, the (scaled) $L^2$ norm is spectrally equivalent to the space-dependent inner product \eqref{space dependent inner product};
thus, the employment of polynomial inverse inequality implies a straightforward relation with the $H^1$ seminorm.
In the VEM framework, it is not possible to proceed similarly for two reasons. The first one is that, at the best of the authors knowledge, inverse inequalities for functions in virtual spaces are not available;
the second reason is that an $L^2$ orthonormal basis of functions in the virtual space is not computable, since such functions are not known explicitly. 
\end{remark}
\subsection{Interspace operators} \label{subsection interspace operator}
In this section, we introduce and construct suitable prolongation and restriction operators acting between the VE spaces $\V_{\ell-1}$ and $\V_{\ell}$, $\ell=p, p-1, \ldots, 2$.
First of all, we stress that the sequence of local spaces $\VpE$, and thus the associated sequence of global spaces $\Vp$, are not nested.
As a consequence, we cannot define the prolongation interspace operator simply as the natural injection, as done for instance in \cite{BrennerScott, BrambleMultigrid, AntoniettiSartiVerani_2015,AXHSV_calcolo}.
In our context, the prolongation operator:
\begin{equation} \label{definition first interspace operator}
\Ipmup: \Vpmu \rightarrow \Vp
\end{equation}
associates to a function $\vpmu$ in $\Vpmu$ a function $\Ipmup \vpmu$ in $\Vp$,
having the same values as $\vpmu$ for all the dofs that are in common with space $\Vpmu$, while the remaining values of the dofs (i.e. the internal higher order ones) are computed using the enhancing constraints presented in definition \eqref{our choice virtual space}.
More precisely, we define $\Ipmup: \Vpmu \rightarrow \Vp$ as:
\begin{equation} \label{definition dofs interspace}
\left\{
\begin{aligned}
\Ipmup \vpmu &= \vpmu,  && \text{on } \partial \E,\\
\int_\E \Ipmup \vpmu \m_{\boldalpha} &= \int_\E \vpmu \m_{\boldalpha}=\dof_{\boldalpha}(\vpmu), && \text{if }\vert \boldalpha \vert \le \p-3,\\
\int_\E \Ipmup \vpmu \m_{\boldalpha} &= \int_\E \Pizpmt \vpmu \m_{\boldalpha}=0, &&\text{if }\vert \boldalpha \vert = \p-2,
\end{aligned}
\right.
\end{equation}
since $\m_{\boldalpha}$ are the elements of an $L^2(E)$-orthonormal basis of $\mathbb P_\p(\E)$.
We recall that the third equation in \eqref{definition dofs interspace} follows from the enhancing constraints in the definition of local spaces \eqref{our choice virtual space}.
The restriction operator $\Ippmu$ is defined as the adjoint of $\Ipmup$ with respect to the space-dependent inner product defined in \eqref{space dependent inner product}, i.e.:
\begin{equation} \label{definition second interspace operator}
(\Ippmu \vp, \wpmu)_{\p-1} = (\vp, \Ipmup \wpmu)_\p \quad \forall \vp \in \Vp, \; \forall \wpmu \in \Vpmu.
\end{equation}
We remark that, thanks to definition \eqref{space dependent inner product} of the space-dependent inner product, the matrix associated with $\Ippmu$ is the transpose of the matrix associated with the operator $\Ipmup$.

\subsection{Smoothing scheme and spectral bounds} \label{subsection smoothing scheme and spectral bounds}
In this section, we introduce and discuss the smoothing scheme entering in the multigrid algorithm.
To this aim, we introduce the following space-dependent norms:
\begin{equation} \label{general discrete norms}
\begin{aligned}
&\vertiii{\vp}_{s,\p}=\sqrt{(\Ap^s \vp, \vp)_\p}
&& \forall \vp \in \Vp, 
&& s\in \mathbb R^+.
\end{aligned}
\end{equation}
We highlight that it holds:
\[
\vertiii{\vp}_{1,\p}^2 = \ap(\vp,\vp).
\]
Since the matrix $\Apbold$ is a symmetric positive definite matrix, there exists an orthonormal (with respect to the inner product $(\cdot, \cdot)_{\p}$) basis of eigenvectors of $\Apbold$, and the associated eigenvalues are real and strictly positive. Let $\{\psi_i, \lambda_i\}_{i=1}^{\dim(\Vp)}$ be the related set of eigenpairs. We show now a bound of the spectrum of $\Apbold$ in terms of $\p$.
\begin{lem} \label{lemma spectrum Ap}
The following upper bound $\Lambda_p$ for the spectrum of $\Apbold$  holds true:
\begin{equation} \label{bound spectrum Ap}
\Lambda_\p \lesssim \frac{\alpha^*(\p)}{\beta_*(\p)},
\end{equation}
where $\alpha^*(\p)$ and $\beta_*(\p)$ are introduced in \eqref{bounds stability constants} and \eqref{bounds space dependent inner product}, respectively.
\end{lem}
\begin{proof}
Let $\lambda_i$ be an eigenvalue of $\Apbold$ and let $\boldsymbol{\psi}_i$ be the associated normalized eigenvector. Then:
\[
\Apbold \cdot \boldsymbol \psi_i = \lambda_i \boldsymbol \psi_i \Longrightarrow (\Ap \psi_i, \psi_i)_\p = \lambda_i (\psi_i, \psi_i)_\p.
\]
Owing to \eqref{stability formula} and \eqref{bounds space dependent inner product}:
\[
\lambda_i = \frac{(\Ap \psi_i, \psi_i)_\p}{(\psi_i, \psi_i)_\p} = \frac{\ap(\psi_i, \psi_i)}{(\psi_i, \psi_i)_\p} \lesssim \alpha^*(\p) \frac{\vert \psi_i \vert^2_{1,\Omega}}{(\psi_i,\psi_i)_{\p}} \lesssim \frac{\alpha^*(\p)}{\beta_*(\p)}.
\]
\end{proof}
As a smoothing scheme, we choose a Richardson scheme, which is given by:
\begin{equation} \label{smoothing operator definition}
\Bp = \widetilde\Lambda_p \cdot \Id_p,
\end{equation}
where $\widetilde\Lambda_\p \le \Lambda_\p$.
A numerical study concerning the (sharp) dependence of $\Lambda_\p$ on $\p$ of the spectral bound $\Lambda_\p$ is presented in Section~\ref{section numerical result}.
\subsection{Error propagator operator} \label{subsection error operator}
As in the classical analysis of the multigrid algorithms \cite{BrennerScott}, in this section we introduce and analyze the error propagator operator.
To this aim, we firstly consider a ``projection'' operator $\Pppmu: \Vp \rightarrow \Vpmu$, defined as the adjoint of $\Ipmup$ with respect to inner product $\ap(\cdot,\cdot)$, i.e.:
\begin{equation} \label{projection operator}
\apmu (\vpmu, \Pppmu \wp) = \ap(\Ipmup \vpmu, \wp) \quad \vpmu\in \Vpmu,\, \wp \in \Vp.
\end{equation}
The following auxiliary result holds.
\begin{lem} \label{lemma projection solution residual equation}
Let $\qpmu \in \Vpmu$ be such that
\[
\Apmu \qpmu  = \rpmu, \quad \text{ with } \rpmu=\Ippmu (\gp - \Ap \zpz),
\]
where $\Ippmu$ is defined in \eqref{definition second interspace operator}, while $\zpz$ is the initial guess of the algorithm and $\Ap$ and $A_{\p-1}$ are defined in \eqref{construction multigrid matrix and rhs}.
Then, it holds:
\begin{equation} \label{projection solution residual equation}
\qpmu = \Pppmu (\zp - \zpz),
\end{equation}
where $\Pppmu$ is defined in \eqref{projection operator}.
\end{lem}
\begin{proof}
As the proof is very similar to its analogous version in \cite[Lemma 6.4.2]{BrennerScott}, here we briefly sketch it. For all $\vpmu \in \Vpmu$:
\[
\begin{split}
& \apmu (\qpmu, \vpmu) =(\Apmu \qpmu, \vpmu)_{\p-1} = (\rpmu, \vpmu)_{\p-1} = (\Ippmu (\gp - \Ap \zpz), \vpmu)_{\p-1}\\
& = (\Ap(\zp - \zpz) , \Ipmup \vpmu)_{\p} = \ap(\zp - \zpz, \Ipmup \vpmu) = \apmu (\Pppmu(\zp - \zpz), \vpmu).
\end{split}
\]
\end{proof}
We now introduce the error propagator operator:
\begin{equation} \label{definition error operator}
\begin{cases}
\mathbb E_{1,\md}  \vp = 0,\\
\mathbb E_{\p,\md} \vp =\left[ \Gpmd  \left( \Id_\p - \Ipmup (\Id_{\p-1} - \mathbb E_{\p-1,\md}^2) \Pppmu  \right)   \right] \vp,
\end{cases}
\end{equation}
where the relaxation operator $\Gp$ is defined as:
\begin{equation} \label{definition relaxation operator}
\Gp = \Id_\p - \Bp^{-1} \Ap,\quad \quad \quad \Bp \text{ being introduced in \eqref{smoothing operator definition}}.
\end{equation}
The following result holds.
\begin{thm} \label{theorem error operator}
Let $\zp$ and $\zp^{(\md+1)}$ be the exact and the multigrid solutions associated with system \eqref{multigrid system}, respectively.
Then, given $\zpz$ initial guess of the algorithm, it holds:
\begin{equation} \label{error operator equality}
\zp - \zp^{(\md+1)} = \mathbb E_{\p,\md}(\zp - \zpz),
\end{equation}
where $\mathbb E_{\p,\md}$ is defined in \eqref{definition error operator}.
\end{thm}
\begin{proof}
We follow the guidelines of \cite[Lemma 6.6.2]{BrennerScott} and proceed by induction.
The initial step of the induction is straightforward since the system is solved exactly at the coarsest level. Therefore, we assume \eqref{error operator equality} true up to $\p-1$ and we prove the claim for $\p$.

Let $\qpmu$, $\epmu$ and $\epmubar$ be introduced in Algorithm \ref{table algorithm}; owing to the induction hypothesis applied to the residual equation, we have:
\[
\qpmu - \epmu = \mathbb E _{\p-1,\md} (\qpmu - \epmubar) = \mathbb E _{\p-1,\md}^2 (\qpmu - 0) = \mathbb E _{\p-1,\md}^2 (\qpmu),
\]
whence:
\begin{equation} \label{residual equality}
\epmu = \qpmu - \mathbb E _{\p-1,\md}^2(\qpmu) = (\Id_{\p-1} - \mathbb E _{\p-1,\md}^2) \qpmu.
\end{equation}
Thus:
\begin{equation} \label{auxiliary step in error operator}
\begin{split}
\zp - \zp^{(\md+1)}	& = \zp - \zp^{(\md)} - \Bp^{-1} (\gp - \Ap \zp^{(\md)})  = (\Id_\p - \Bp^{-1}\Ap) (\zp - \zp^{(\md)})\\
				& = (\Id_\p - \Bp^{-1}\Ap) ^{\md} (\zp - \zp^{(1)}) = \Gpmd (\zp - \zp^{(1)}) = \Gpmd (\zp-\zpz - \Ipmup \epmu).
\end{split}
\end{equation}
Inserting \eqref{projection solution residual equation} and \eqref{residual equality} in \eqref{auxiliary step in error operator}, we get:
\[
\begin{split}
\zp - \zp^{(\md+1)} 	& = \Gpmd \left( \zp - \zpz - \Ipmup(\Id_{\p-1} - \mathbb E _{\p-1,\md}^2) \Pppmu (\zp-\zpz)  \right) \\
				& =  \Gpmd \left( \Id_\p - \Ipmup(\Id_{\p-1} - \mathbb E _{\p-1,\md}^2) \Pppmu  \right) (\zp-\zpz).\\
\end{split}
\]
\end{proof}

\section{Convergence analysis of the multigrid algorithm} \label{section convergence analysis of the multigrid method}
We prove in Section~\ref{subsection convergence of the multilevel algorithm} the convergence of the multigrid algorithm presented in Section~\ref{section multigrid methods with non inherited sublevel solvers}.
For the purpose, we preliminarily introduce some technical tools.
In Section~\ref{subsection generalized Cauchy Schwarz inequality and smoothing property}, we discuss the so-called smoothing property associated with the Richardson scheme \eqref{smoothing operator definition}.
In Section~\ref{section bounds on operators}, we show bounds related to the prolongation operator $\Ipmup$ \eqref{definition first interspace operator} and its adjoint with respect to the space-dependent inner product \eqref{space dependent inner product}.
Bounds concerning the error correction steps are the topic of Section~\ref{subsection error correction step bounds}.
Finally, in Sections \ref{subsection convergence of the two level algorithm} and \ref{subsection convergence of the multilevel algorithm},
we treat the convergence of the two-level  and multilevel algorithm, respectively.

\subsection{Smoothing property} \label{subsection generalized Cauchy Schwarz inequality and smoothing property}
\begin{lem} \label{lemma smoothing property}
(smoothing property)
For any $\vp \in \Vp$, it holds that:
\begin{equation} \label{smoothing estimates}
\begin{split}
& \vertiii{\Gpmd \vp} _{1,\p} \le \vertiii{\vp}_{1,\p},\\
& \vertiii{\Gpmd \vp} _{s,\p} \lesssim \left( \frac{\alpha^*(\p)}{\beta_*(\p)}   \right)^{\frac{s-t}{2}}(1+\md)^{\frac{t-s}{2}} \vertiii{\vp}_{t,\p},
\end{split}
\end{equation}
for some $0\le t\le s\le 2$, $\md\in \mathbb N \setminus{\{0\}}$, where $\alpha^*(\p)$ and $\beta_*(\p)$ are defined in \eqref{bounds stability constants} and in \eqref{bounds space dependent inner product}, respectively.
\end{lem}
\begin{proof}
The proof is analogous to that in \cite[Lemma 4.3]{AntoniettiSartiVerani_2015}. For the sake of clarity, we report the details.
To start with, we rewrite $\vp$ in terms of the orthonormal basis of eigenvectors $\{\psi_i\}_{i=1}^{\dim(\Vp)}$ of $\Apbold$ as follows:
\[
\vp = \sum_{i=1}^{\dim(\Vp)} \v_i \psi_i,\quad \forall \vp \in \Vp.
\]
As a consequence,
\[
\Gpmd \vp = \left(\Id_p  - \frac{1}{\Lambda_\p} \Ap \right)^{\md} \vp = \sum_{i=1}^{\dim(\Vp)} \left(1- \frac{\lambda_i}{\Lambda_\p} \right)^{\md} \v_i \psi_i,
\]
where $\Lambda_\p$ is the upper bound for the spectrum of $\Apbold$ presented in Lemma \ref{lemma spectrum Ap}.
Then, owing to the orthonormality of $\psi_i$ with respect to the inner product $(\cdot, \cdot)_\p$, we have:
\[
\begin{split}
& \vertiii{\Gpmd \vp}^2_{s,\p} = \sum_{i=1}^{\dim(\Vp)} \left(   1-\frac{\lambda_i}{\Lambda_\p} \right)^{2\md} \v_i^2 \lambda_i^s
			= \Lambda _\p^{s-t} \sum_{i=1}^{\dim(\Vp)} \left( 1- \frac{\lambda_i}{\Lambda_\p}  \right)^{2\md} \frac{\lambda_i^{s-t}}{\Lambda_\p^{s-t}} \lambda_i^t \v_i^2\\
& \le \Lambda_\p^{s-t} \max_{x\in [0,1]} (x^{s-t} (1-x)^{2\md}) \vertiii{\vp}^2_{t,\p} \lesssim \left( \frac{\alpha^*(\p)}{\beta_*(\p)}  \right)^{s-t} (1+\md)^{t-s} \vertiii{\vp}^2_{t,\p},
\end{split}
\]
where in the last inequality we used \cite[Lemma 4.2]{AntoniettiSartiVerani_2015} and \eqref{bound spectrum Ap}.
\end{proof}

\subsection{Prolongation and projection operators} \label{section bounds on operators}
In this section, we prove bounds in the $\vertiii{\cdot}_{1,\p}$ norm of the prolongation and the projection operators defined in \eqref{definition first interspace operator} and \eqref{projection operator}, respectively.
We stress that this set of results deeply relies on the \emph{new} enhancing stategy presented in the definition of the virtual space \eqref{our choice virtual space}.

We start with a bound on the prolongation operator.
\begin{thm} \label{theorem stability interspace operator}
(bound on the prolongation operator) There exists $\cstab$, positive constant independent of the discretization and multigrid parameters, such that:
\begin{equation} \label{stability interspace operator}
\vertiii{\Ipmup \vpmu}_{1,\p} \le \cstab \sqrt{\frac{\alpha^*(\p) \beta^*(\p)}{\alpha_*(\p) \beta_*(\p)}} \vertiii{\vpmu}_{1,\p-1} \quad \forall \vpmu \in \Vpmu,
\end{equation}
where $\alpha_*(\p)$, $\alpha^*(\p)$ are introduced in \eqref{bounds stability constants} whereas $\beta_*(\p)$ and $\beta^*(\p)$ are introduced in \eqref{bounds space dependent inner product}.
\end{thm}
\begin{proof}
Recalling bounds \eqref{stability formula}, \eqref{stability bounds space dependent inner product} and the definition of the auxiliary space-dependent inner product \eqref{auxiliary space dependent inner product}, we have:
\begin{equation} \label{aux step c}
\begin{split}
& \vertiii{\Ipmup \vpmu}^2_{1,\p} = \sum_{\E \in \tauh} \vertiii{\Ipmup \vpmu}^2_{1,\p;\E} = \sum_{\E \in \tauh} \apE(\Ipmup \vpmu, \Ipmup \vpmu)\\
& \lesssim \alpha^*(\p) \a(\Ipmup \vpmu, \Ipmup \vpmu) \lesssim \frac{\alpha^*(\p)}{\beta_*(\p)} (\Ipmup \vpmu,  \Ipmup \vpmu)_{\p}.
\end{split}
\end{equation}
We recall that:
\[
(\Ipmup \vpmu, \Ipmup \vpmu)_\p = \sum_{j=1}^{\dim(\Vp)} \dof_j^2(\Ipmup \vpmu).
\]
Since $\{\mathbb B_{\p}(\partial \E)\}_{\p=1}^{+\infty}$ defined in \eqref{boundary space} is a sequence of \emph{nested} space for all $\E \in \tauh$, we directly have:
\[
\dof_{b,j}^2 (\Ipmup \vpmu) = \dof_{b,j}^2(\vpmu),
\]
where $\dof _{b,j}(\cdot)$ denotes the $j$-th boundary dof.\\

Now, we deal with the internal degrees of freedom.
We cannot use the above nestedness argument since the sequence $\{\Vp\}_{\p=1}^{\dim(\Vp)}$ is made of non-nested spaces.
In order to overcome this hindrance, recalling the definition of the prolongation operator \eqref{definition dofs interspace}, we write:
\[
\dof_{i,j}(\Ipmup \vpmu) = \frac{1}{\vert \E \vert} \int_\E \Ipmup \vpmu \m_{\boldalpha} = \begin{cases}
\frac{1}{\vert \E \vert} \int_\E \vpmu \m_{\boldalpha} & \text{if } \vert \boldalpha\vert \le \p-3,\\
0 & \text{if } \vert \boldalpha\vert  =  \p-2,
\end{cases}
\]
where $\dof_{i,j}(\cdot)$ denotes the $j$-th internal dof. As a consequence, it holds:
\begin{equation} \label{killing prolongation}
(\Ipmup \vpmu, \Ipmup \vpmu)_\p = (\vpmu , \vpmu)_{\p-1} = \vertiii{\vpmu}_{0,\p-1}^2.
\end{equation}
Next, we relate $\vertiii{\cdot}_{0,\p-1}$ with $\vertiii{\cdot}_{1,\p-1}$. We note that:
\begin{equation} \label{bounding L2 with H1 discrete}
\vertiii{\vpmu}^2_{0,\p-1} \lesssim \beta^*(\p) \vert \vpmu \vert^2_{1,\E} \lesssim \frac{\beta^*(\p)}{\alpha_*(\p)} \vertiii{\vpmu}_{1,\p-1}^2,
\end{equation}
where we used in the last but one and in the last inequalities \eqref{bounds space dependent inner product} and \eqref{stability formula}, respectively.

Combining \eqref{aux step c}, \eqref{killing prolongation} and \eqref{bounding L2 with H1 discrete}, we get the claim.
\end{proof}

We show an analogous bound for the ``projection'' operator $\Pppmu$ introduced in \eqref{projection operator}.
\begin{thm} \label{theorem stability projection operator}
(bound on the ``projection'' operator)
There exists $\cstab$, positive constant independent of the discretization and multigrid parameters, such that:
\begin{equation} \label{stability projection operator}
\vertiii{\Pppmu \vp}_{1,\p-1} \le \cstab \sqrt{\frac{\alpha^*(\p) \beta^*(\p)}{\alpha_*(\p) \beta_*(\p)}} \vertiii{\vp}_{1,\p} \quad \forall \vp \in \Vp,
\end{equation}
where $\alpha_*(\p)$, $\alpha^*(\p)$ are introduced in \eqref{bounds stability constants} whereas $\beta_*(\p)$ and $\beta^*(\p)$ are introduced in \eqref{bounds space dependent inner product}.
The constant $\cstab$ is the same constant appearing in the statement of Theorem \ref{theorem stability interspace operator}
\end{thm}
\begin{proof}
It suffices to note that:
\[
\vertiii{\Pppmu \vp} _{1,\p-1} = \max_{\wpmu \in \Vpmu \setminus{\{0\}}} \frac{\apmu (\Pppmu \vp, \wpmu)}{\vertiii{\wpmu}_{1,\p-1}} = \max_{\wpmu \in \Vpmu \setminus{\{0\}}} \frac{\ap (\vp, \Ipmup \wpmu)}{\vertiii{\wpmu}_{1,\p-1}}
\]
and then apply Theorem \ref{theorem stability interspace operator} along with a Cauchy-Schwarz inequality.
\end{proof}

\subsection{Error correction step} \label{subsection error correction step bounds}
In this section, we prove a bound for the error correction step discussed in the multigrid algorithm, see Table \ref{table pth level}.
\begin{thm} \label{theorem approximation property}
(bound on the error correction step) There exists a positive constant $\c$ independent of the discretization parameters such that:
\begin{equation} \label{approximation property}
\vertiii{(\Id_p - \Ipmup \Pppmu) \vp}_{0,\p} \le \c \frac{\alpha^*(\p)}{\alpha_*(\p)^{\frac{3}{2}}} \frac{\beta^*(\p)^{\frac{3}{2}}}{\beta_*(\p)} \vertiii{\vp}_{1,\p} \quad \forall \vp \in \Vp,
\end{equation}
where $\alpha_*(\p)$, $\alpha^*(\p)$ are introduced in \eqref{bounds stability constants} whereas $\beta_*(\p)$ and $\beta^*(\p)$ are introduced in \eqref{bounds space dependent inner product}.
\end{thm}
\begin{proof}
Applying \eqref{stability formula} and \eqref{bounds space dependent inner product}, we have:
\[
\begin{split}
&\vertiii{(\Id_\p - \Ipmup \Pppmu) \vp}^2_{0,\p} \lesssim \beta^*(\p) \left \vert (\Id_\p - \Ipmup \Pppmu) \vp \right \vert_{1,\Omega}^2 \\
& \lesssim \beta^*(\p) \alpha_*(\p)^{-1} \sum_{\E \in \tauh}\left\{ \apE((\Id_\p- \Ipmup \Pppmu) \vp, (\Id_\p- \Ipmup \Pppmu) \vp  )  \right\}.
\end{split}
\]
Therefore, we deduce:
\[
\begin{split}
& \vertiii{(\Id_\p - \Ipmup \Pppmu)\vp}^2_{0,\p} \\
& \lesssim \beta^*(\p) \alpha_*(\p)^{-1} \sum_{\E \in \tauh} \left\{ \apE(\vp, \vp) + \apE(\Ipmup \Pppmu \vp, \Ipmup \Pppmu \vp) - 2 \apE(\vp, \Ipmup \Pppmu \vp)    \right\}\\
& = \beta^*(\p) \alpha_*(\p)^{-1} \sum_{\E \in \tauh} \left\{ \apE(\vp, \vp) + \apE(\Ipmup \Pppmu \vp, \Ipmup \Pppmu \vp) - 2 \apmuE(\Pppmu \vp, \Pppmu \vp)  \right\}\\
& \lesssim \beta^*(\p) \alpha_*(\p)^{-1} \sum_{\E \in \tauh} \left\{ \apE(\vp,\vp) + \frac{\alpha^*(\p) \beta^*(\p)}{\alpha_*(\p) \beta_*(\p)} \apE (\Pppmu \vp, \Pppmu \vp)   \right\},
\end{split}
\]
where in the last inequality we applied Theorem \ref{theorem stability interspace operator} and we dropped the third term since it is negative.
Finally, applying Theorem \ref{theorem stability projection operator}, we obtain:
\[
\vertiii {(\Id_\p -\Ipmup \Pppmu) \vp}^2_{0,\p} \lesssim \frac{\alpha^*(\p)^2}{\alpha_*(\p)^3} \frac{\beta^*(\p)^3}{\beta_*(\p)^2} \vertiii{\vp}^2_{1,\p},
\]
whence the claim.
\end{proof}
\subsection{Convergence of the two-level algorithm} \label{subsection convergence of the two level algorithm}
In this section, we prove the convergence of the two-level algorithm.
\begin{thm} \label{theorem convergence of the two level algorithm}
There exists a positive constant $\cdlvl$ independent of the discretization and multilevel parameters, such that:
\begin{equation} \label{two level contraction}
\vertiii{\mathbb E_{\p,\md}^{\text{2lvl}}\vp }_{1,\p} \le \cdlvl \Sigmapmd \vertiii{\vp}_{1,\p} \quad \forall \vp \in \Vp,
\end{equation}
where
\[
\Sigmapmd = \left( \frac{\alpha^*(\p) \beta^*(\p)}{\alpha_*(\p) \beta_*(\p)} \right)^{\frac{3}{2}} \cdot \frac{1}{\sqrt{1+\md}}
\]
and $\mathbb E_{\p,\md}^{\text{2lvl}}$ is the two-level error propagator operator:
\[
\mathbb E_{\p,\md}^{\text{2lvl}} \vp =\left[ \Gpmd  \left( \Id_\p - \Ipmup \Pppmu  \right)   \right] \vp.
\]
The constants $\alpha_*(\p)$ and $\alpha^*(\p)$ are introduced in \eqref{bounds stability constants}, whereas the constants $\beta_*(\p)$ and $\beta^*(\p)$ are introduced in \eqref{bounds space dependent inner product}.
\end{thm}
\begin{proof}
Using the smoothing property \eqref{smoothing estimates} and Theorem \ref{theorem approximation property}, we get:
\[
\begin{split}
&\vertiii{\mathbb E_{\p,\md}^{\text{2lvl}} \vp} _{1,\p}	= \vertiii{\Gpmd (\Id_\p - \Ipmup \Pppmu)\vp}_{1,\p} \lesssim \frac{1}{\sqrt{1+\md}}\cdot  \sqrt{\frac{\alpha^*(\p)}{\beta_*(\p)}} \vertiii{(\Id_\p - \Ipmup \Pppmu)\vp}_{0,\p}\\
& \lesssim \frac{1}{\sqrt{1+\md}}\cdot  \sqrt{\frac{\alpha^*(\p)}{\beta_*(\p)}} \cdot  \frac{\alpha^*(\p)}{\alpha_*(\p)^{\frac{3}{2}}} \cdot  \frac{\beta^*(\p)^{\frac{3}{2}}}{\beta_*(\p)}\vertiii{\vp}_{1,\p}
	=  \left(\frac{\alpha^*(\p) \beta^*(\p)}{\alpha_*(\p) \beta_*(\p)} \right)^{\frac{3}{2}} \cdot \frac{1}{\sqrt{1 + \md}} \vertiii{\vp}_{1,\p}.
\end{split}
\]
\end{proof}
As a consequence of Theorem \ref{theorem convergence of the two level algorithm}, we deduce that taking $\md$, number of postsmoothing iterations large enough,
the two-level algorithm converges, since the two-level error propagator operator $\mathbb E_{\p,\md}^{\text{2lvl}}$ is a contraction.
We point out that a sufficient condition for the convergence of the two-level algorithm is that the number of postsmoothing iterations $\md$ must satisfy:
\begin{equation} \label{number postsmoothing 2 level}
\sqrt{1+\md} > \cdlvl^{-1} \left( \frac{\alpha^*(\p) \beta^*(\p)}{\alpha_*(\p) \beta_*(\p)}\right)^{\frac{3}{2}},
\end{equation}
see Remark \ref{remark post smoothing step} for more details.
We stress that \eqref{number postsmoothing 2 level} is a sufficient condition only, in practice
the number of postsmoothing steps needed for the convergence of the algorithm is much smaller; see numerical results in Section~\ref{section numerical result}.


\subsection{Convergence of the multilevel algorithm} \label{subsection convergence of the multilevel algorithm}
In this section, we prove the main result of the paper, namely the convergence of our $\p$-VEM multigrid algorithm.
\begin{thm} \label{theorem convergence of the multilevel algorithm}
Let $\Sigmapmd$ and $\cdlvl$ be defined as in Theorem \ref{theorem convergence of the two level algorithm}. Let $\cstab$ be defined as in Theorem \ref{theorem stability projection operator}.
Let $\alpha_*(\p)$ and $\alpha^*(\p)$ be defined in \eqref{stability formula} and $\beta_*(\p)$ and $\beta^*(\p)$ be defined in \eqref{stability bounds space dependent inner product}.
Then, there exists $\chat > \cdlvl$ such that, if the number of postsmoothing iterations satisfies:
\begin{equation} \label{assumption post smoothing convergence multigrid}
\sqrt{1 +\md} > \frac{\cstab^2 \chat^2}{\chat-\cdlvl} \left( \frac{\alpha^*(\p) \beta^*(\p)}{\alpha_*(\p) \beta_*(\p)}  \right)^{\frac{5}{2}} ,
\end{equation}
then, it holds:
\[
\Vert \mathbb E _{\p,\md} \vp \Vert_{1,\p} \le \chat \Sigmapmd \Vert \vp \Vert_{1,\p},
\]
with $\chat \Sigmapmd <1$.
As a consequence, this implies that the multilevel algorithm converges uniformly with respect to the discretization parameters and the number of levels provided that $\md$ satisfies \eqref{assumption post smoothing convergence multigrid},
since $\mathbb E_{\p,\md}$ is a contraction.
\end{thm}
\begin{proof}
From Theorem \ref{theorem convergence of the two level algorithm}, we have that:
\[
\Sigmapmd = \left( \frac{\alpha^*(\p) \beta^*(\p)}{\alpha_*(\p) \beta_*(\p)}\right)^{\frac{3}{2}} \cdot  \frac{1}{\sqrt{1+\md}}.
\]
Recalling \eqref{definition error operator}, we decompose the error propagator operator as:
\[
\mathbb E_{p,\md} \vp = \Gp^{\md} (\Id_\p  - \Ipmup \Pppmu )\vp + \Gp^{\md} \Ipmup \mathbb E_{\p-1,\md}^2 \Pppmu \vp = \mathbb E_{\p,\md}^{\text{2lvl}} \vp + \Gp^{\md} \Ipmup \mathbb E_{\p-1,\md}^2 \Pppmu \vp.
\]
Thus:
\[
\vertiii{\mathbb E_{\p,\md} \vp} _{1,\p} \le  \vertiii{\mathbb E_{\p,\md}^{\text{2lvl}} \vp} _{1,\p} + \vertiii{\Gp^{\md} \Ipmup \mathbb E_{\p-1,\md}^2 \Pppmu \vp} _{1,\p} = I + II.
\]
We bound the two terms separately.
The first one is estimated directly applying the two-level error result, namely Theorem \ref{theorem convergence of the two level algorithm}:
\[
I \le \cdlvl \Sigmapmd \vertiii{\vp}_{1,\p}.
\]
On the other hand, the second term can be bounded applying the smoothing property Lemma \eqref{lemma smoothing property},
the bounds regarding the interspace operator Theorem \ref{theorem stability interspace operator}, the induction hypothesis
and Theorem \ref{theorem stability projection operator}.
We can write:
\[
\begin{split}
& II \le \vertiii{\Ipmup \mathbb E_{\p-1,\md}^2 \Pppmu \vp}_{1,\p} \le \cstab \sqrt{\frac{\alpha^*(\p) \beta^*(\p)}{\alpha_*(\p) \beta_*(\p)}} \vertiii{\mathbb E_{\p-1,\md}^2 \Pppmu \vp }_{1,\p-1}\\
& \le \cstab \sqrt{\frac{\alpha^*(\p) \beta^*(\p)}{\alpha_*(\p) \beta_*(\p)}} \chat ^2\Sigmapmumd^2 \vertiii{\Pppmu \vp}_{1,\p-1} \le \cstab^2 \chat^2  \frac{\alpha^*(\p) \beta^*(\p)}{\alpha_*(\p) \beta_*(\p)} \Sigmapmumd \vertiii{\vp}_{1,\p}.
\end{split}
\]
We note that owing to \eqref{bounds stability constants} and \eqref{bounds space dependent inner product}, the following holds true:
\[
\Sigmapmumd^2 =  \left( \frac{\alpha^*(\p-1) \beta^*(\p-1)}{\alpha_*(\p-1) \beta_*(\p-1)} \right)^3 \cdot \frac{1}{1+\md} \approx
		\left( \frac{\alpha^*(\p) \beta^*(\p)}{\alpha_*(\p) \beta_*(\p)} \right)^{\frac{3}{2}} \cdot \frac{1}{\sqrt{1+\md}} \Sigmapmd^2.
\]
We deduce:
\[
\vertiii{\mathbb E_{\p,\md} \vp}_ {1,\p} \le 
\underbrace{\left( \cdlvl + \cstab^2 \chat^2 \left( \frac{\alpha^*(\p) \beta^*(\p)}{\alpha_*(\p) \beta_*(\p)} \right)^{\frac{5}{2}} \cdot \frac{1}{\sqrt{1+\md}}   \right) \Sigmapmd}_{\bbeth} \vertiii{\vp}_{1,\p}.
\]
We want that $\bbeth$ is such that $\bbeth < \chat \Sigmapmd$. In particular,  we require:
\[
\cdlvl + \cstab^2 \chat^2 \left( \frac{\alpha^*(\p) \beta^*(\p)}{\alpha_*(\p) \beta_*(\p)} \right)^{\frac{5}{2}}\cdot \frac{1}{\sqrt{1+\md}}  < \chat,
\]
which is in fact equivalent to \eqref{assumption post smoothing convergence multigrid}.
\end{proof}

\begin{remark} \label{remark post smoothing step}
We briefly comment on equations \eqref{number postsmoothing 2 level} and \eqref{assumption post smoothing convergence multigrid} highlighting the origin of the different terms:
\begin{itemize}
\item[*] 
the term $\frac{\alpha^*(\p)}{\alpha_*(\p)} \approx \p^{10}$ originates from the spectral property \eqref{stabilization abstract} of the stabilization term $\SE$;
if it were possible to provide a discrete bilinear form \eqref{stability formula} with continuity and coercivity constants provably independent of $\p$, then $\frac{\alpha^*(\p)}{\alpha_*(\p)} \approx 1$;
\item[*] the term $\frac{\beta_*(\p)}{\beta_*(\p)} \approx \p^{6}$ is related to \eqref{bounds space dependent inner product} which is not $\p$ robust;
again, if it were possible to provide space-dependent inner products spectrally equivalent to the $H^1$ seminorm, then $\frac{\beta_*(\p)}{\beta_*(\p)} \approx 1$.
\end{itemize}
The existence of a $\p$ independent stabilization of the method and the existence of a computable virtual $L^2$-orthonormal basis is still, at the best of the authors knowledge, an open issue.\\
\end{remark}

\section{Numerical results} \label{section numerical result}
In this section, we test the performance of the multigrid solver for the $\p$-version of the VEM by studying the behaviour of the
convergence factor:
\begin{equation} \label{convergence factor}
\rho = \exp \left( \frac{1}{N} \ln \left( \frac{\Vert r_N \Vert_2}{\Vert r_0 \Vert_2}  \right)  \right),
\end{equation}
where $N$ denotes the iteration counts needed to reduce the residual below a given tolerance  of $10^{-8}$ and $r_N$, $r_0$ are the final and the initial residuals, respectively. We also show that our multigrid algorithm can be employed as a preconditioner for the PCG method. 
Throughout the section we fix the maximum number of iterations to $1000$ and consider three different kind of decompositions: meshes made of squares, Voronoi-Lloyd polygons and quasi-regular hexagons; cf. Figure \ref{meshes employed}.
\begin{figure}  [h]
\centering
\subfigure {\includegraphics [angle=0, width=0.32\textwidth]{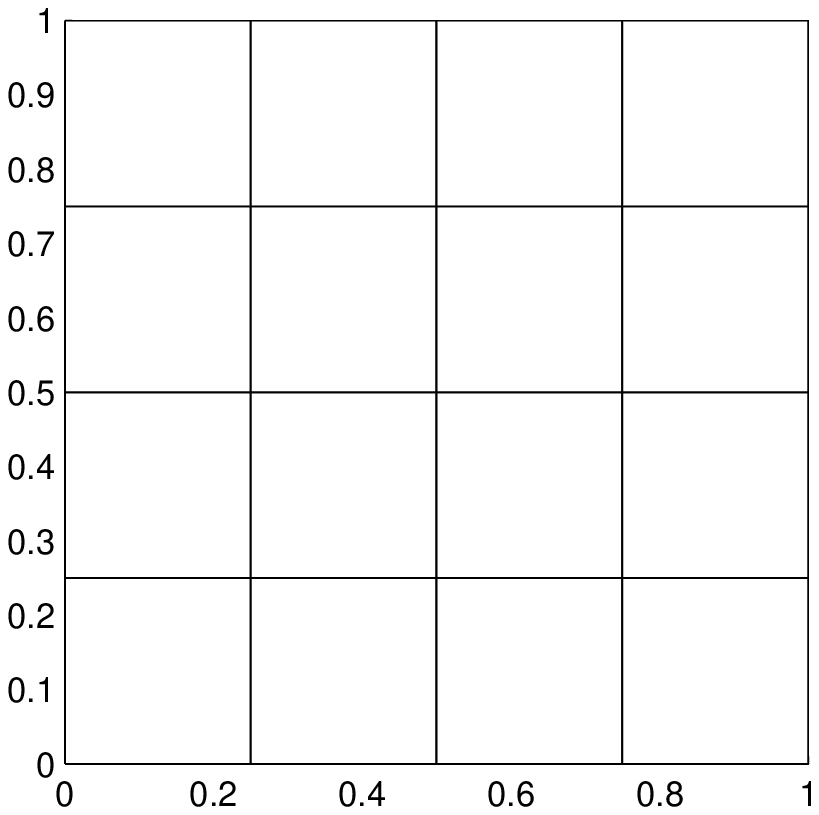}}
\subfigure {\includegraphics [angle=0, width=0.32\textwidth]{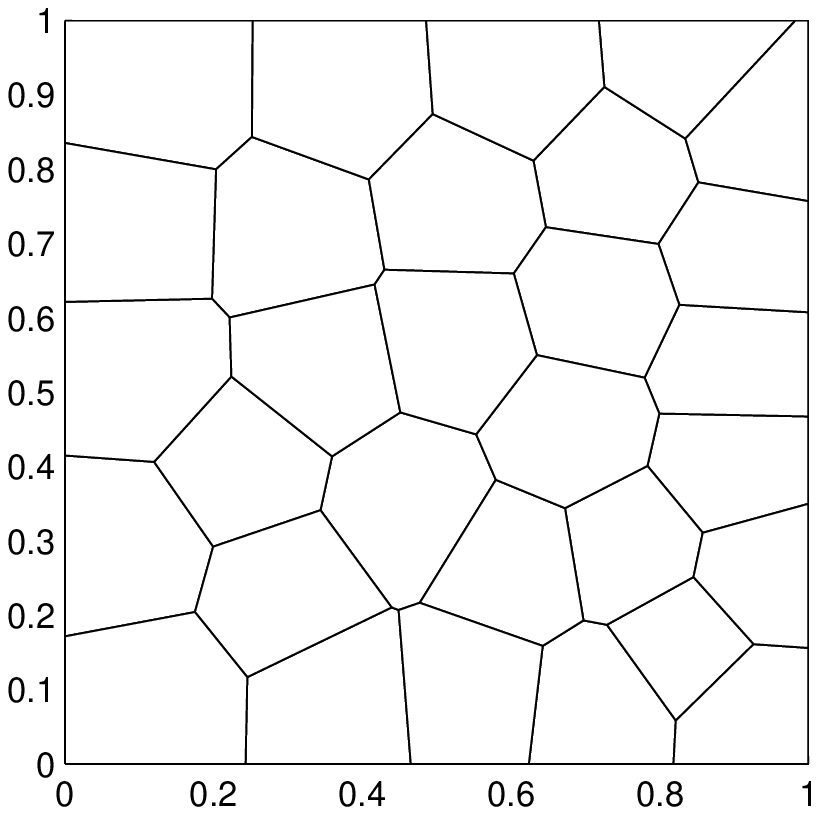}}
\subfigure {\includegraphics [angle=0, width=0.32\textwidth]{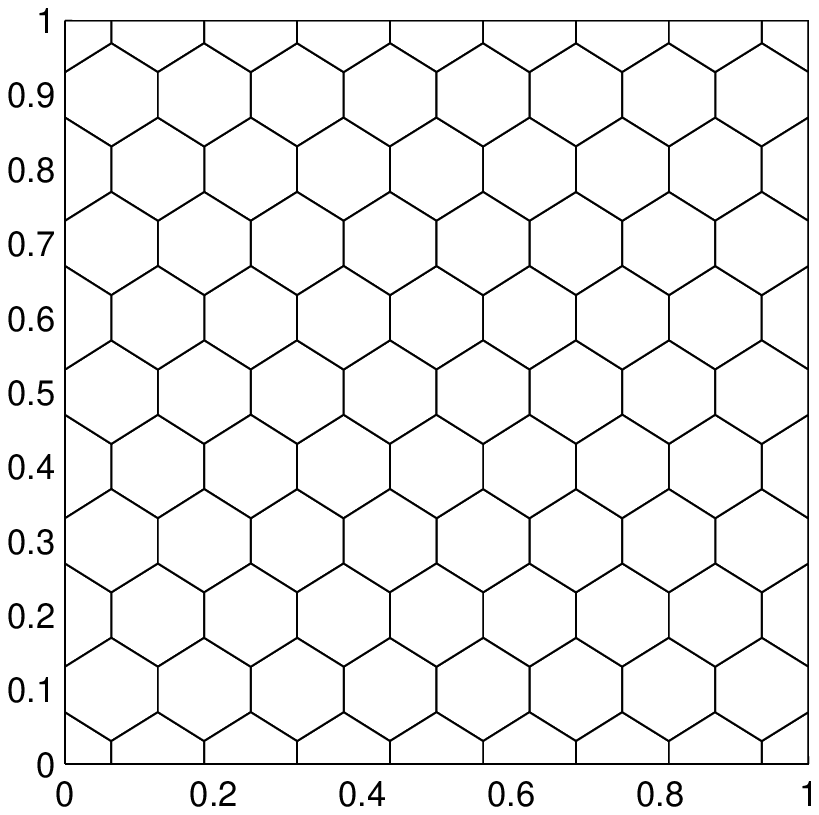}}
\caption{Meshes made of: squares (left), Voronoi-Lloyd polygons (centre), quasi-regular hexagons (right).} \label{meshes employed}
\end{figure}
In Section~\ref{sec:MG}, we present some tests aiming at assessing the performance of our multigrid scheme with different smoothers. In Section~\ref{sec:PCG} we show that our multigrid method can be successfully employed as a preconditioner for the Congiugate Gradient (CG) iterative scheme, more precisely we consider a single iteration of the multigrid algorithm as a preconditioner to accelerate the Preconditioned CG method.

\subsection{The $\p$--multigrid algorithm as an iterative solver} \label{sec:MG}
In this section we investigate the performance of our multigrid scheme with different smoothers. We consider both the Richardson scheme \eqref{smoothing operator definition} as well as  a symmetrized Gau\ss-Seidel scheme as a smoother.

The first set of numerical experiment has been obtained based on employing a Richardson smoother. Before presenting the computed estimates of the convergence factor, we investigate numerically the behaviour of the smoothing parameter $\Lambda_\p$ associated with the Richardson scheme \eqref{smoothing operator definition}, for which a far-from-being-sharp bound is given in Lemma \ref{lemma spectrum Ap}.
As shown in Figure \ref{figure spectral upper bound}, where $\Lambda_\p$ as a function of $\p$ is shown, the maximum eigenvalue of $\Kpbold$
seems to behave even better than $\p^{2}$, which is the expected behaviour in standard Finite Elements. The numerical tests presented in the following have been obtained with
an approximation of $\Lambda_\p$ obtained (in a off line stage) with ten iterations of the power method.
\begin{figure}  [h]
\centering
\includegraphics [angle=0, width=0.7\textwidth]{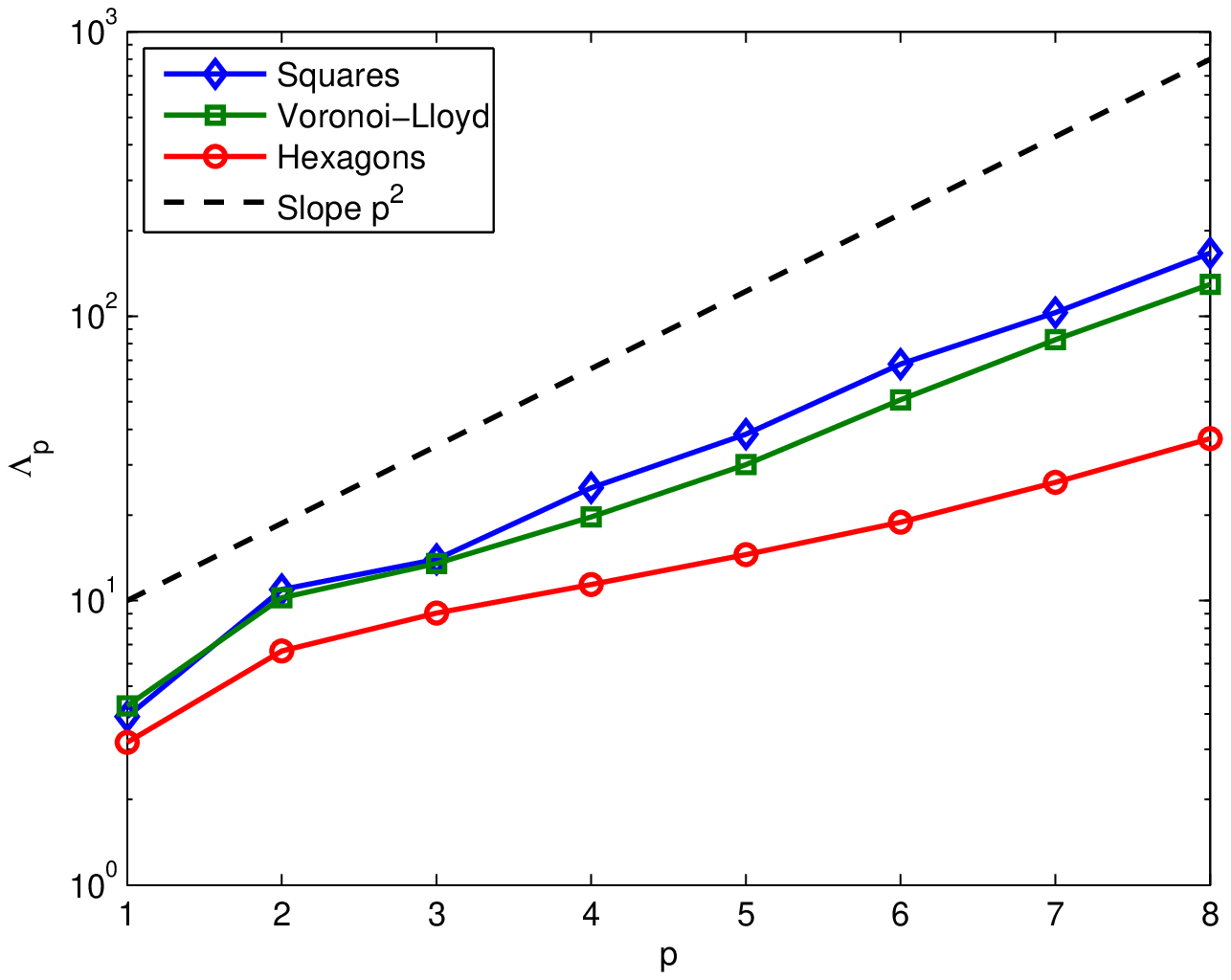}
\caption{Maximum eigenvalue $\Lambda_\p$ of $\Kpbold$  as a function of $\p$.}
\label{figure spectral upper bound}
\end{figure}
In this section, we numerically investigate the behaviour of the multigrid algorithm using a Richardson smoother. 
The results reported in Table \ref{tabular squares} shows the computed convergence factor defined in $\rho$ \eqref{convergence factor} as a function of the number of level $K$, the number of postsmoothing steps $\md=m$, 
and the degree of accuracy $\p$ employed at the ``finest level'' on a mesh made of squares, cf. Figure \ref{meshes employed}.
Analogous results have been obtained on the other decompositions; such results are not reported here for the sake of brevity.
As expected, increasing the number of postsmoothing $\md$ implies a decreasing of the convergence factor $\rho$. Moreover, a minimum number of smoothing steps is required to guarantee the convergence of the underlying solver.
We also observe that, as expected, even though both two-level and multilevel solvers converge for a fixed value of $m$, the number of iterations required to reduce the relative residual below the given tolerance grows with increasing $p$.
\begin{table}[!htpb]
\centering
\begin{tabular}{|c|c|cc|cc|cc|}
\multicolumn{1}{c|}{} &	\multicolumn{1}{c|}{$p=2$} 	& \multicolumn{2}{c|}{$p=3$} & \multicolumn{2}{c|}{$p=4$}
& \multicolumn{2}{c|}{$p=5$} \\
\hline
K & 2 & 2 & 3 & 3 & 4& 3 & 4\\	 	
 \hline
\multicolumn{1}{|c|}{$\md=2$} & 0.99 & x     & 0.97 & x      & 0.97 & x       & x\\
\multicolumn{1}{|c|}{$\md=4$} & 0.97 & x      & 0.95 & x      & 0.92 & x       & x\\
\multicolumn{1}{|c|}{$\md=6$} & 0.96 & 0.93 & 0.92 & 0.79 & 0.88 & x      & 0.85\\
\multicolumn{1}{|c|}{$\md=8$} & 0.95 & 0.69 & 0.89  & 0.74 & 0.84& 0.98 & 0.82\\
\hline
\end{tabular}
	\caption {Convergence factor $\rho$ of the $\p$--multigrid scheme as a function of $K$  (number of levels), $p$ (``finest'' level) 
	and $\md$ (number of postsmoothing steps). Richardson smoother. Mesh of squares.} 
	\label{tabular squares}
\end {table}
A numerical estimate of the minimum number of postsmoothing steps needed in practice to achieve convergence is reported in Table \ref{tabular minimum number of iterations} for all the meshes depicted in Figure \ref{meshes employed}. This represents a \emph{practical} indication for \eqref{assumption post smoothing convergence multigrid}
As expected, such a minimum number depends on the polynomial degree  employed in the finest level.
\begin{table}[H]
\centering
	\begin{tabular}{|c|c|cc|ccc|ccc|ccc|}
\multicolumn{1}{c|}{} 
&	\multicolumn{1}{c|}{$p=2$}  
& \multicolumn{2}{c|}{$p=3$}  
& \multicolumn{3}{c|}{$p=4$}  
& \multicolumn{3}{c|}{$p=5$} 
& \multicolumn{3}{c|}{$p=6$} \\
\hline
K & 2 & 2 & 3 & 2 & 3 & 4 & 2 & 3 & 4 & 2 & 3 & 4 \\
\hline
\text{Square}           & 1 & 6 & 1 & 10 & 5    & 1   & 14 & 8 & 5 & 42 & 15 & 8\\
\text{Voronoi-Lloyd} & 7 & 14 & 5 & 12 & 11 & 5 & 14 & 10 & 11 & 36 & 24 & 9\\
\text{Hexagons}      & 7 & 25 & 6 & 12 & 20 & 5 & 9 & 10 & 19 & 17 & 7 & 9\\
\hline
\end{tabular}
\caption {Minimum number of postsmoothing steps needed to guarantee convergence.} 
\label{tabular minimum number of iterations}
\end {table}

We next investigate the behaviour of our MG algorithm whenever a symmetrized Gau\ss-Seidel scheme as a smoother is employed. 
We recall that the smoothing matrix $\Bp$  associated with the  symmetrized Gau\ss-Seidel operator  now reads:
\begin{equation} \label{Gauss Lobatto scheme}
\mathbf{B_{\p}}= \begin{cases} \Lpbold \text{ if the postsmoothing iteration is odd}\\ 
\mathbf L_{\mathbf \p}^T \text{ if } \text{ if the postsmoothing iteration is even}\\ \end{cases}
\end{equation}
where $\Lpbold $ is the lower triangular part of $\Kpbold$. We have repeated the set of experiments carried out before employing the same same set of parameters: the results are are shown in Tables \ref{tabular squares GS}, \ref{tabular voronoilloyd GS}. 
As expected, employing a symmetrized Gau\ss-Seidel smoother yields to an iterative scheme whose convergence factor
is smaller than in the analogous cases with the Richardson smoother. In Table \ref{tabular voronoilloyd GS} we report the same results obtained on a 
mesh of Voronoi-Lloyd polygonal elements keep on increasing the number of post smoothing steps: as expected the performance of the algorithm improves further. The same kind of results have been obtained on a regular hexagonal grid; for the sake of brevity these results have been omitted.
\begin{table}[H]
\centering
\begin{tabular}{|c|c|cc|cc|cc|}
\multicolumn{1}{c|}{} &	\multicolumn{1}{c|}{$p=2$} 	& \multicolumn{2}{c|}{$p=3$} & \multicolumn{2}{c|}{$p=4$}
& \multicolumn{2}{c|}{$p=5$} \\
\hline
K & 2 & 2 & 3 & 3 & 4& 3 & 4 \\
		\hline
		m= 2      &  0.96 & 0.90 & 0.92  & x      & 0.75 & 0.97  & x      \\
		m= 4      &  0.92 & 0.69 & 0.85  & 0.57 & 0.57 & 0.72  & x      \\
		m = 6     &  0.88 & 0.60 & 0.78  & 0.43 & 0.44 & 0.60 & 0.85 \\
		m = 8     &  0.84 & 0.53 & 0.72  & 0.34 & 0.35 & 0.53 & 0.82\\
		\hline
		\end{tabular}
	\caption {Convergence factor $\rho$ of the $\p$--multigrid scheme as a function of $K$  (number of levels), $P$ (``finest'' level) and $\md$ (number of postsmoothing steps). Gau\ss-Seidel smoother. Mesh of squares.} \label{tabular squares GS}
\end {table}
\begin{table}[H]
\centering
\begin{tabular}{|c|c|cc|cc|cc|}
\multicolumn{1}{c|}{} 
&	\multicolumn{1}{c|}{$p=2$} 	
& \multicolumn{2}{c|}{$p=3$} 
& \multicolumn{2}{c|}{$p=4$}
& \multicolumn{2}{c|}{$p=5$} \\
\hline
K & 2 & 2 & 3 & 3 & 4& 3 & 4 \\
\hline
$m= 8$      & 0.91 & 0.63 & 0.81 &  0.45 & 0.61 & 0.49 & 0.46\\
$m= 10 $   & 0.89 & 0.57 & 0.77 &  0.37 & 0.54 &  0.44 & 0.43 \\
$m = 12$   & 0.87 & 0.52 & 0.73 &  0.31 & 0.47 & 0.40 & 0.40 \\
$m = 14$   & 0.86 & 0.48 & 0.69 &  0.25 & 0.42 &  0.37 & 0.37 \\
\hline
\end{tabular}
\caption {Convergence factor $\rho$ of the $\p$--multigrid scheme as a function of $K$  (number of levels), $p$ (``finest'' level) and $\md$ (number of postsmoothing steps). Gau\ss-Seidel smoother. Mesh of Voronoi-Lloyd polygons.} \label{tabular voronoilloyd GS}
\end {table}

\subsection{The $\p$--multigrid algorithm as a preconditioner for the PCG method} \label{sec:PCG}
In this set of experiments we aim at demonstrating that  \emph{a single iteration} of the $\p$--multigrid algorithm can be successfully employed  to precondition the CG method. In this set of experiments, the coarsest level is given by $\p=1$. In all the test cases, we have empoyed as a stopping criterion in order to reduce the (relative) residual below a tolerance of $10^{-6}$, with a maximum number of iterations set equal to $1000$.
In Figure \ref{figure PCG}, we compare the PCG iteration counts with our multigrid preconditioner, which is constructed employing either a Richardson or a Gau\ss-Seidel smoother and $m=8$ post-smoothing steps. For the sake of comparison, we report the same quantities computed with the unpreconditioned CG method and with the PCG method with preconditioner given by an incomplete Cholesky factorization. 
As before, the results reported in Figure \ref{figure PCG} have been obtained on the computational grids depicted in Figure \ref{meshes employed}. 
\begin{figure}  [h]
\centering
\subfigure {\includegraphics [angle=0, width=0.49\textwidth]{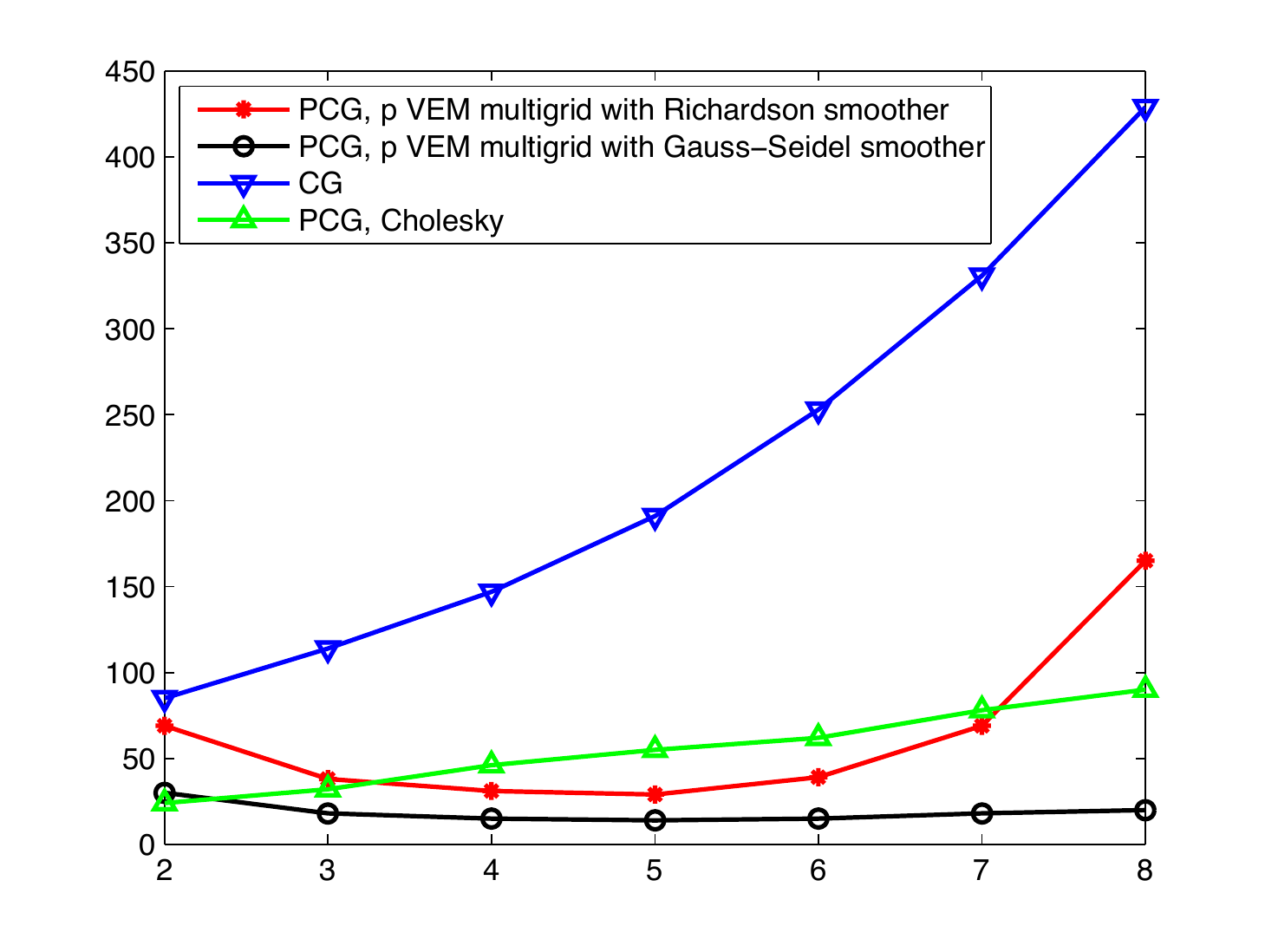}}
\subfigure {\includegraphics [angle=0, width=0.49\textwidth]{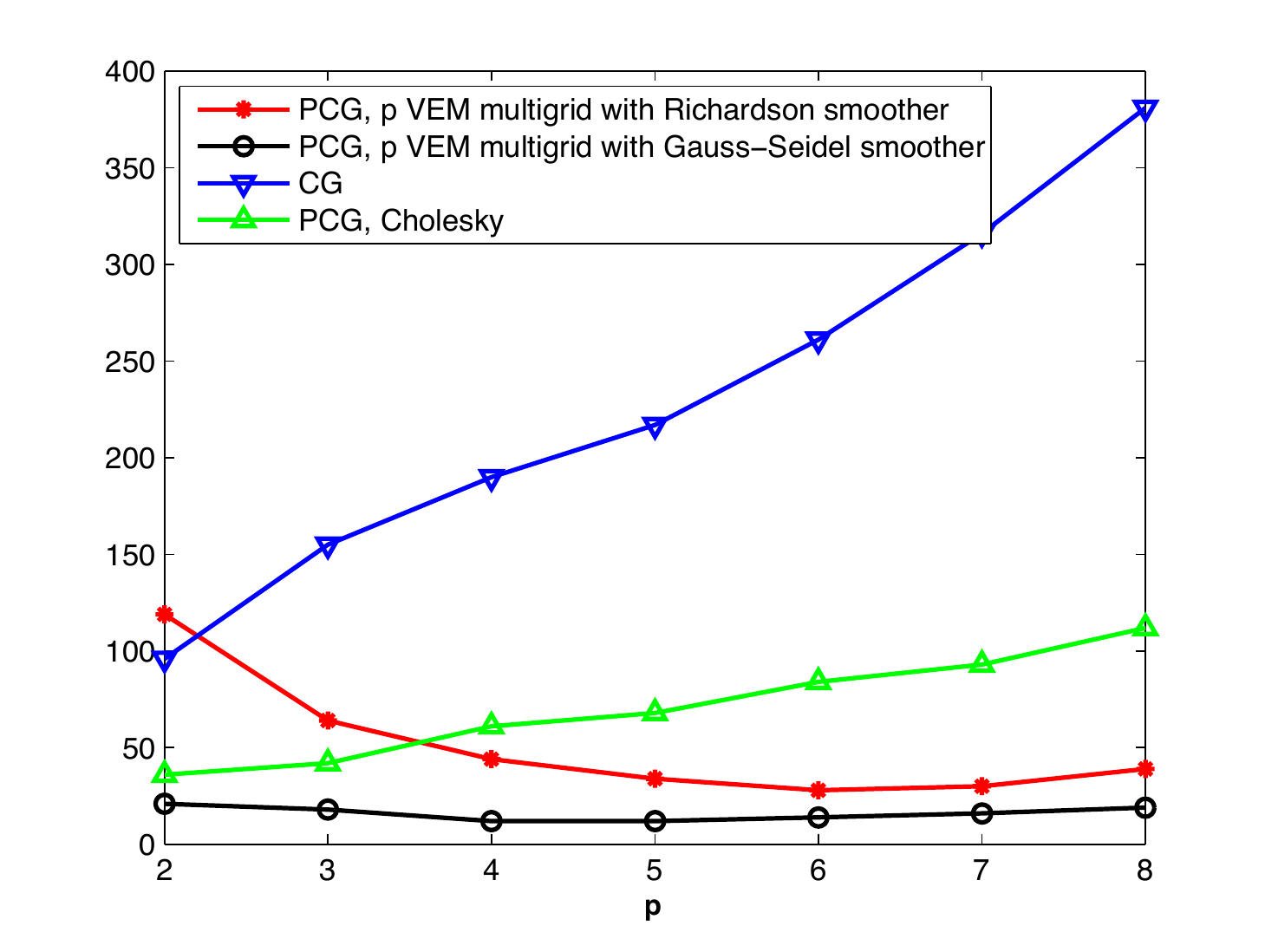}}
\caption{PCG iteration counts as a function of $p$ with $p$--multigrid preconditioner (with either Richardson or Gau\ss-Seidel smoothers). 
For the sake of comparison the CG iteration counts without preconditioning and with an incomplete Cholesky preconditioner are also shown. 
For the $p$--multigrid preconditioner, the coarsest level is $\p=1$  and the number of post-smoothing steps is $8$.
Meshes made of: Voronoi-Lloyd polygons (left), quasi-regular hexagons (right).} \label{figure PCG}
\end{figure}
From Figure \ref{figure PCG}, we infer that PCG iteration counts needed to reduce the residual below a given tolerance seems to be almost 
constant whenever the multigrid preconditioner with Gau\ss-Seidel smoother  is employed, even for a relatively small number of smoothing steps.
In contrast, as expected, the incomplete Cholesky preconditioner does not provide a uniform preconditioner.
Also,  the multigrid preconditioner with Richardson smoother seems to perform well at least on regular hexagonal grids.

\bibliographystyle{abbrv}
\bibliography{biblio_paola,bibliogr}

\end{document}